\documentclass{amsart}
\usepackage{amsmath}
\usepackage{amsfonts}
\usepackage{amssymb}
\usepackage{ascmac}
\usepackage{layout}
\usepackage{amsthm}
\usepackage{graphicx}
\usepackage[svgnames]{xcolor}
\usepackage{tikz}

\usepackage{url}

\title{Walks along a weak square sequence and the non-semiproperness of Namba forcings}

\author{Kenta Tsukuura}

\makeatletter
\@namedef{subjclassname@2020}{%
  \textup{2020} Mathematics Subject Classification}
\makeatother

\address{Department of Fisheries Distribution and Management, National Fisheries University, Shimonoseki, 759-6533, Japan}
\email{tsukuura@fish-u.ac.jp} 

\thanks{This research was supported by JSPS KAKENHI Grant Number 24K16959. The author is grateful to Toshimichi Usuba and Hiroshi Sakai for helpfull discussions. The author acknowledge the use of ChatGPT in refining the English presentation of this paper.}

\newcommand{\force}{\Vdash}

\theoremstyle{plain}
\newtheorem{thm}{Theorem}[section]

\newtheorem{rema}[thm]{Remark}
\newtheorem{lem}[thm]{Lemma}
\newtheorem{coro}[thm]{Corollary}
\newtheorem{clam}[thm]{Claim}
\newtheorem{prop}[thm]{Proposition}
\newtheorem{ques}[thm]{Question}

\begin{document}

\maketitle
\markright{WALKS ALONG A $\square$ SEQUENCE AND THE NON-SEMIPROPERNESS OF $\mathrm{Nm}$}
\begin{abstract}
In this paper, we demonstrate that if, for every $\kappa$-complete fine filter $F$ over $\mathcal{P}_{\kappa}\lambda$, the associated Namba forcing $\mathrm{Nm}(\kappa,\lambda,F)$ is semiproper, then $\square(\mu,{<}\aleph_1)$ fails for all regular $\mu \in [\lambda, 2^{\lambda}]$ under the certain cardinal arithmetic. In particular, this result establishes that the consistency strength of the semiproperness of $\mathrm{Nm}(\aleph_2,F)$ for every $\aleph_2$-complete filter $F$ over $\aleph_2$ exceeds the strength of infinitely many Woodin cardinals.

Minimal walk methods associated with a square sequece play a central role in this paper. These observations introduce two-cardinal walks with naive $C$-sequences and show that the existence of non-reflecting stationary subsets implies $\mathcal{P}_{\kappa}\lambda \not\to [I_{\kappa\lambda}^{+}]^{3}_{\lambda}$.
\end{abstract}

\section{Introduction}\label{sec:intro}

The notion of strong compactness has been widely studied in the contexts of infinitary logic, large cardinals, and infinitary combinatorics. The original definition is due to Tarski. For a regular cardinal $\kappa$, the infinitary logic $\mathcal{L}_{\kappa\kappa}$ allows conjunctions of ${<}\kappa$-many formulas and the use of ${<}\kappa$-many quantifiers to define formulas. The classical first-order logic $\mathcal{L}_{\omega\omega}$ satisfies the compactness theorem. For an uncountable cardinal $\kappa$, $\kappa$ is strongly compact if and only if $\mathcal{L}_{\kappa\kappa}$ satisfies the compactness theorem. 

As is widely known, strong compactness has several characterizations. The following are equivalent:
\begin{enumerate}
    \item $\kappa$ is strongly compact.
    \item $\mathcal{P}_\kappa \lambda$ carries a fine ultrafilter for all $\lambda \geq \kappa$.
    \item Every $\kappa$-complete filter $F$ can be extended to a $\kappa$-complete ultrafilter, i.e., there exists a $\kappa$-complete ultrafilter $U$ such that $F \subseteq U$.
    \item The two-cardinal tree property $\mathrm{TP}(\kappa,\lambda)$ holds for all $\lambda \geq \kappa$, and $\kappa$ is inaccessible.
\end{enumerate}

In~\cite{MR3959249}, Hayut analyzed these relationships:
\begin{thm}[Hayut~\cite{MR3959249}]\label{hayut}
For regular cardinals $\kappa \leq \lambda$, the following are equivalent: 
\begin{enumerate}
    \item Every $\kappa$-complete filter over $\lambda$ can be extended to a $\kappa$-complete ultrafilter. 
    \item $\mathcal{L}_{\kappa\kappa}$ satisfies the compactness theorem for a theory of size $2^{\lambda}$, meaning that for every $\mathcal{L}_{\kappa\kappa}$-theory $T$ with $2^{\lambda}$ symbols, if $T$ is ${<}\kappa$-consistent, then $T$ is consistent. 
\end{enumerate}
In particular, if every $\kappa$-complete filter over $\lambda$ can be extended to a $\kappa$-complete ultrafilter, then $\square(\mu,{<}\kappa)$ fails for all $\mu \in [\lambda,2^{\lambda}] \cap \mathrm{Reg}$. 
\end{thm}

Hayut also pointed out the following implications concerning the compactness of $\kappa$:
\begin{center}
\begin{tabular}{cccc}
 & $2^\lambda$-strongly compact & $\Rightarrow$ & $\lambda$-compact\\
 $\Leftrightarrow$ & $\mathcal{L}_{\kappa\kappa}$ satisfies the comp. thm. for $T$ of size $2^{\lambda}$ & $\Rightarrow$ & $\lambda$-strongly compact
\end{tabular}
\end{center}
Here, $\lambda$-strong compactness asserts the existence of a fine ultrafilter over $\mathcal{P}_{\kappa}(\lambda)$. $\lambda$-compactness corresponds to (1) in Theorem~\ref{hayut}. 

The principle $(\dagger)$, which asserts that every $\omega_1$-stationary preserving poset is semiproper, was introduced in~\cite{MR924672}. Shelah proved that $(\dagger)$ is equivalent to semistationary reflection $\mathrm{SSR}$~\cite{MR1623206}. Many combinatorial consequences of $(\dagger)$ were discovered by Todor\v{c}evi\'{c}~\cite{MR0657114}, Sakai~\cite{MR2387938}, and Sakai--Veli\u{c}kovi\'{c}~\cite{MR3284479}. 

Shelah also proved that $(\dagger)$ is equivalent to the assertion that, for every $\kappa$, the Namba forcing $\mathrm{Nm}(\kappa)$ is semiproper. The author introduced two-cardinal variations of Namba forcings $\mathrm{Nm}(\kappa,\lambda,F)$ with a filter $F$ over $\mathcal{P}_{\kappa}\lambda$ in~\cite{tsukuura} to study $(\dagger)$. As illustrated in Figure~\ref{fig:hierarchy}, a hierarchy of semiproperness for Namba forcings was established. Note that the right-hand side of Figure~\ref{fig:hierarchy} remains meaningful if $\kappa$ is a successor cardinal. 

In this paper, we first reexamine Theorem~\ref{hayut} from a combinatorial perspective. We then demonstrate an analogous conclusion in the context of Namba forcing:

\begin{thm}\label{maintheorem}
For regular cardinals $\aleph_2 \leq \kappa \leq \lambda = {\lambda^{<\kappa}}$, suppose that $\mathrm{Nm}(\kappa,\lambda,F)$ is semiproper for all $\kappa$-complete filters $F$ over $\mathcal{P}_{\kappa}\lambda$. Then $\square(\mu,{<}\aleph_1)$ fails for all regular $\mu \in [\lambda,2^{\lambda}]$. In particular, both $\square(\lambda,{<}\aleph_1)$ and $\square(2^{\lambda},{<}\aleph_1)$ fail. 
\end{thm}

The consistency strength of Namba forcing $\mathrm{Nm}(\aleph_2)$ remains an open problem, specifically whether it requires the existence of a measurable cardinal. However, Theorem~\ref{maintheorem} establishes that the consistency strength of the assertion that $\mathrm{Nm}(\aleph_2,F)$ is semiproper for every $\aleph_2$-complete filter $F$ over $\aleph_2$ is greater than that of the existence of infinitely many Woodin cardinals (See Corollary \ref{coro:consistencystrength}). 
\begin{center}
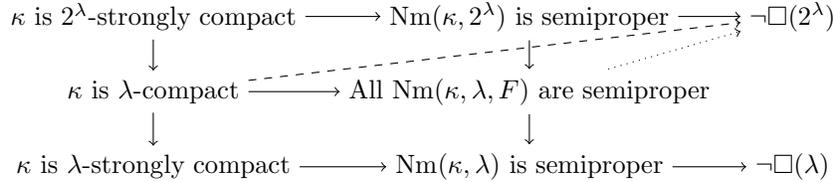
\begin{figure}[bthp]
\centering
\begin{tikzpicture}
\node (sc1) at (-5,0) {$\kappa$ is $2^{\lambda}$-strongly compact};
\node (sc2) at (-5,-1) {$\kappa$ is $\lambda$-compact};
\node (sc3) at (-5,-2) {$\kappa$ is $\lambda$-strongly compact};

\node (nm1) at (0,0) {$\mathrm{Nm}(\kappa, 2^{\lambda})$ is semiproper};
\node (nm2) at (0,-1) {All $\mathrm{Nm}(\kappa,\lambda, F)$ are semiproper};
\node (nm3) at (0,-2) {$\mathrm{Nm}(\kappa, \lambda)$ is semiproper};

\node (sq1) at (3.5,0) {$\lnot \square(2^{\lambda})$};
\node (sq2) at (3.5,-2) {$\lnot \square(\lambda)$};

\draw[->] (sc1) to (sc2);
\draw[->] (sc2) to (sc3);
\draw[->] (nm1) to (nm2);
\draw[->] (nm2) to (nm3);

\draw[->] (sc1) to (nm1);
\draw[->] (sc2) to (nm2);
\draw[->] (sc3) to (nm3);

\draw[->] (nm1) to (sq1);
\draw[->] (nm3) to (sq2);
\draw[->, dashed] (sc2) to (sq1);

\draw[->, dotted] (nm2) to (sq1);
\end{tikzpicture}
\caption{The ``dashed'' arrow and the ``dotted'' arrow are Theorems~\ref{hayut} and~\ref{maintheorem}, respectively.}
\label{fig:hierarchy}
\end{figure}
\end{center}

In~\cite{tsukuura}, we asked whether the statements ``$\mathrm{Nm}(\kappa,\lambda)$ is semiproper'' and ``All $\mathrm{Nm}(\kappa,\lambda,F)$ are semiproper'' can be separated. Theorem~\ref{maintheorem} provides a positive answer to this question.

\begin{thm}\label{maintheorem:model}
If $\kappa$ is $\lambda$-strongly compact, then for every regular cardinal $\aleph_1 \leq \mu < \kappa$, there exists a $\mu$-strategically closed poset $P$ that forces:
\begin{enumerate}
    \item $([\kappa,\lambda] \cap \mathrm{Reg})^{V} = [\kappa,\lambda] \cap \mathrm{Reg}$ and $\mu^{+} = \kappa$. 
    \item $\mathrm{Nm}(\kappa,\lambda,F)$ is semiproper for some $\kappa$-complete fine filter $F$ over $\mathcal{P}_{\kappa}\lambda$. 
    \item $\mathrm{Nm}(\kappa,\lambda,F')$ is \emph{not} semiproper for some $\kappa$-complete fine filter $F'$ over $\mathcal{P}_{\kappa}\lambda$. 
\end{enumerate}
\end{thm}

To prove Theorem~\ref{maintheorem}, we introduce a specific filter $F_{\overline{d}}$ associated with a thin $\mathcal{P}_{\kappa}\mu$-list. This filter provides a combinatorial proof of Theorem~\ref{hayut} by establishing the following result. The implication (2) $\to$ (4) is proved in Section~\ref{sec:filter}. After presenting this proof, we adapt the method to establish Theorem~\ref{maintheorem}. 

\begin{thm}\label{maintheorem2}
For regular cardinals $\kappa \leq \lambda \leq \mu \leq 2^{\lambda}$, the following are equivalent:
\begin{enumerate}
    \item Every $\kappa$-complete filter over $\lambda$, generated by $\mu$-many sets, can be extended to a $\kappa$-complete ultrafilter.
    \item Every $\kappa$-complete fine filter over $\mathcal{P}_{\kappa}\lambda$, generated by $\mu$-many sets, can be extended to a $\kappa$-complete ultrafilter.
    \item $\mathcal{L}_{\kappa\kappa}$ satisfies the compactness theorem for a theory of size $\mu$.
    \item $\mathrm{TP}(\kappa,\mu)$ holds, and $\kappa$ is inaccessible. 
\end{enumerate}
\end{thm}

Throughout this paper, walks along a $\square(\mu,{<}\kappa)$-sequence play a central role. However, in Sections~\ref{sec:reflections} and~\ref{sec:twocardinalwalks}, we also utilize other $C$-sequences. Beyond establishing non-semiproperness, we also study other principles. 

One of the strongest negations of the Ramsey property arises in the context of large cardinals. It is well known that the weak compactness of an uncountable cardinal $\kappa$ is characterized by the Ramsey property $\kappa \to (\kappa)_2$. However, an open problem remains as to whether the strong compactness of such a $\kappa$ can be characterized by the two-cardinal Ramsey property $\forall \lambda \geq \kappa (\mathrm{Part}(\kappa,\lambda))$.

Here, $\mathrm{Part}(\kappa,\lambda)$ asserts that for every coloring $ c: \{\{x,y\} \subseteq \mathcal{P}_{\kappa} \lambda \mid x \subset y\} \to 2 $, there exists an unbounded subset $ H \subseteq \mathcal{P}_{\kappa} \lambda $ such that $ c $ is constant on all pairs $\{x, y\}$ with $x \subset y$ from $H$.

Todor\v{c}evi\'{c}~\cite{partitionctblesets} introduced two-cardinal walks and used them to construct a coloring that strongly negates $\mathrm{Part}(\kappa, \lambda)$ for $\kappa = \aleph_1$. Later, Velleman~\cite{velleman} obtained a similar result without relying on walks. Todor\v{c}evi\'{c}~\cite{MR2355670} further extended this approach by providing such a coloring for successors of regular cardinals $\kappa$. Additionally, Shioya~\cite{shioya2005partitioning} established analogous results for uncountable $\kappa$ satisfying $\kappa^{\mathrm{cf}(\lambda)} = \kappa$.

In this paper, we introduce a new type of two-cardinal walk and use it to prove a strong negation of the two-cardinal Ramsey property.

\begin{thm}\label{maintheorem3}
For regular cardinals $\kappa \leq \lambda$, if there exists a stationary subset $S \subseteq E_{<\kappa}^{\lambda}$ such that $S \cap \alpha$ is non-stationary for all $\alpha \in E_{<\kappa}^{\lambda}$, then $\mathcal{P}_{\kappa}\lambda \not\to [I_{\kappa\lambda}^{+}]^{3}_{\lambda}$. 
\end{thm}

The paper is structured as follows:

\begin{itemize}
    \item \textbf{Section~\ref{sec:pre}:} We review Namba forcings, minimal walk methods, and two-cardinal combinatorics. In particular, we describe how to destroy the semiproperness of Namba forcings using a set related to Galvin games. 
    \item \textbf{Section~\ref{sec:filter}:} We introduce a $\kappa$-complete fine filter $F_{\overline{d}}$ over $\lambda$ using a thin list $\overline{d}$ and study its properties. This section includes the proof of Theorem~\ref{maintheorem2}.
    \item \textbf{Section~\ref{sec:walksquare}:} We develop minimal walk methods along a square sequence, focusing on the \emph{non}-semiproperness of Namba forcings as preparation for Theorem~\ref{maintheorem}.
    \item \textbf{Section~\ref{sec:maintheorem}:} This section is devoted to proving Theorem~\ref{maintheorem} and its applications. One application concerns consistency strength, while another establishes Theorem~\ref{maintheorem:model}.
    \item \textbf{Section~\ref{sec:reflections}:} We discuss reflection principles, including stationary reflection and the singular cardinal hypothesis. 
    \item \textbf{Section~\ref{sec:twocardinalwalks}:} We introduce a ``naive'' $C$-sequence (see Remark~\ref{rema:twocardinalwalks}) and explore its applications. Theorem~\ref{maintheorem3} is proven in this section.
    \item \textbf{Section~\ref{sec:ques}:} We analyze the dissection of $(\dagger)$ from the perspective of Namba forcings. The paper concludes with several open questions.
\end{itemize}

\section{Preliminaries}\label{sec:pre}

In this section, we recall basic facts about Namba forcing with filters, minimal walk methods, and two-cardinal combinatorics. We use~\cite{MR1994835} as a general reference for set theory. For topics related to Namba forcing, we refer to~\cite[Section XII]{MR1623206} and~\cite{tsukuura}. The latter is a previous work by the author. For minimal walk methods, we refer to~\cite{MR2355670}. 

Our notation is standard. In this paper, $\kappa$ and $\lambda$ denote regular cardinals greater than $\aleph_2$ unless otherwise stated. $\mathcal{P}_{\kappa}\lambda$ is the set of all $x \subseteq \lambda$ with $|x| < \kappa$. We use $\mu$ to denote an infinite cardinal. For $\kappa < \lambda$, $E^{\lambda}_\kappa$ and $E^{\lambda}_{<\kappa}$ denote the set of all ordinals below $\lambda$ with cofinality $\kappa$ and $<\kappa$, respectively. We use $\xi, \zeta, \eta, \dots$ and $\alpha, \beta, \gamma, \dots$ to denote ordinals. We also write $[\kappa,\lambda]$ for the interval of ordinals $\{\xi \mid \kappa \leq \xi \leq \lambda\}$. By $\mathrm{Lim}$ and $\mathrm{Reg}$, we mean the classes of limit ordinals and regular cardinals, respectively. For a set $X$ of ordinals, $\mathrm{Lim}(X)$ is the set of all ordinals $\alpha < \sup X$ with $\sup(X \cap \alpha) = \alpha$. For $\subseteq$-increasing finite sequences $s,t \in [\mathcal{P}(X)]^{<\omega}$, we write $s \sqsubseteq t$ to mean that $t$ end-extends $s$, that is, $\forall a \in t \setminus s \forall b \in s (b \subseteq a)$. 

Our notation for filters and ideals follows~\cite{MR2768692}. We consider a filter $F$ over a set $Z$ such that $Z \subseteq \mathcal{P}(X)$ for some $X$. If $X = \lambda$, then $\lambda$, $\mathcal{P}_{\kappa}\lambda$, and $[\lambda]^{\kappa}$ are typical examples of $Z$. The set of $F$-positive sets is denoted by $F^{+}$. We say that an ideal $I$ is fine if $\{z \in Z \mid x \in z\} \in F$ for all $x \in X$. For an ideal $I$ over $Z$, $I^{\ast}$ is the dual filter of $I$. The completeness number $\mathrm{comp}(F)$ of $F$ is the least $\delta$ such that $F$ is not $\delta^{+}$-complete. 

A cardinal $\kappa$ is $\lambda$-strongly compact if it is a regular cardinal such that $\mathcal{P}_{\kappa}\lambda$ carries a fine ultrafilter. $\kappa$ is $\lambda$-compact if, for every $\kappa$-complete fine filter $F$ over $\lambda$, there exists a $\kappa$-complete fine ultrafilter $U$ over $\lambda$ such that $F \subseteq U$. 

In this paper, $\theta$ always denotes a sufficiently large regular cardinal. Given an expansion $\mathcal{A} = \langle \mathcal{H}_{\theta},\in,\dots\rangle$ and a set $X \subseteq \mathcal{H}_\theta$, we denote by $\mathrm{Sk}_{\mathcal{A}}(X)$ the least elementary substructure of $\mathcal{A}$ containing $X$ as a subset. If $\mathcal{A}$ is clear from the context, we simply write $\mathrm{Sk}(X)$ for $\mathrm{Sk}_{\mathcal{A}}(X)$. The following lemma is frequently used without explicit reference in this paper. 

\begin{lem}
For an expansion $\mathcal{A} = \langle \mathcal{H}_{\theta},\in,\dots\rangle$ and an $M \prec \mathcal{A}$, for every pair of sets $a$ and $X$ with $a \in X \in M$, we have
\[
\mathrm{Sk}(M \cup \{a\}) = \{f(a) \mid f:X \to \mathcal{H}_{\theta}, f \in M\}.
\]
In particular,
\[
\mathrm{Sk}(M \cup \{a\}) \cap \omega_1 = \{f(a) \mid f:X \to \omega_1, f \in M\}.
\]
\end{lem}

\subsection{Namba forcings}\label{subsec:namba}
Namba forcing, introduced by Namba~\cite{MR0297548}, is a poset that forces $\mathrm{cf}(\aleph_2) = \omega$ while preserving $\aleph_1$. We begin by introducing a class, which includes Namba's original poset, denoted by $\mathrm{Nm}(Z,F)$, indexed by filters $F$ and its base set $Z$. For an $\aleph_2$-complete fine filter $F$ over $Z \subseteq \mathcal{P}(X)$, an $F$-Namba tree $p$ is a set $p \subseteq [Z]^{<\omega}$ satisfying the following conditions:
\begin{itemize}
 \item $p$ is a tree, i.e., each $s \in p$ is $\subseteq$-increasing, and $p$ is closed under initial segments.
 \item There exists a maximal $s \in p$ such that $\forall t \in p(s \subseteq t \lor t \subseteq s)$. This $s$ is called the trunk, denoted by $\mathrm{tr}(p)$.
 \item For each $s \in p$, if $s \sqsupseteq \mathrm{tr}(p)$, then $\mathrm{Suc}(s) = \{a \in Z \mid s{^{\frown}}\langle a\rangle \in p\} \in F^{+}$.
\end{itemize}

Let $\mathrm{Nm}(Z, F)$ denote the set of all $F$-Namba trees. If $Z = \mathcal{P}_{\kappa}\lambda$, we write $\mathrm{Nm}(\kappa,\lambda,F)$ to refer to $\mathrm{Nm}(Z,F)$ (to emphasize the context of two-cardinal combinatorics). For a filter $F$, $\mathrm{Nm}(Z, F)$ is defined as $\mathrm{Nm}(Z, F^{\ast})$. For an ideal $I$, we consider its dual filter, i.e., $\mathrm{Nm}(Z,I)$ is $\mathrm{Nm}(Z,I^{\ast})$. We write $\mathrm{Nm}(\kappa)$ and $\mathrm{Nm}(\kappa, \lambda)$ to refer to $\mathrm{Nm}(\kappa,I_{\kappa})$ and $\mathrm{Nm}(\mathcal{P}_{\kappa}\lambda, I_{\kappa\lambda})$, respectively. Here, $I_\kappa$ and $I_{\kappa\lambda}$ are bounded ideals over $\kappa$ and $\mathcal{P}_{\kappa}\lambda$, respectively. $\mathrm{Nm}(\aleph_2)$ is the original Namba forcing. We now examine the basic properties of $\mathrm{Nm}(Z,F)$. 

\begin{lem}\label{lem:namba1}
 If $F$ is an $\aleph_2$-complete fine filter over $Z\subseteq \mathcal{P}(X)$ and $\kappa = \mathrm{comp}(F)$, then the following holds:
\begin{enumerate}
 \item $\mathrm{Nm}(Z,F) \force \dot{\mathrm{cf}}(\kappa) = \omega$. 
 \item For all $\delta \in [\kappa,|X|^{V}]\cap \mathrm{Reg}$, if $\{z \in Z \mid |X| < z\} \in F$, then $\mathrm{Nm}(Z,F) \force \dot{\mathrm{cf}}(\delta) = \omega$. 
 \item For all $\delta \in [\aleph_1,\kappa) \cap \mathrm{Reg}$, if $\delta^{\omega} < \kappa$ or $\delta = \aleph_1$, then for every semistationary (or stationary, respectively) subset $S \subseteq [\delta]^{\omega}$, $\mathrm{Nm}(Z,F) \force S$ is semistationary (or stationary, respectively). In particular, $\mathrm{Nm}(Z,F) \force \dot{\mathrm{cf}}(\delta) > \omega$. 
 \item $\mathrm{Nm}(Z,F)$ is $\omega_1$-stationary preserving. 
\end{enumerate}
\end{lem}

\begin{proof}
 We verify (1). By $\kappa = \mathrm{comp}(F)$, there exists a partition $Z = \bigcup_{\alpha < \kappa}Z_{\alpha}$ such that each $Z_{\alpha}$ belongs to the dual ideal $F^\ast$. Let $\dot{g}$ be an $\mathrm{Nm}(Z,F)$-name for the set $\bigcup \{\mathrm{tr}(t) \mid t \in \dot{G}\}$. It is easy to see that $\mathrm{Nm}(Z,F) \force \sup_{n}\{\sup \{\alpha \mid \forall \beta < \alpha(\dot{g}_{n} \not\in Z_{\beta})\}\} = \kappa$. 

 Next, we prove (2). We may assume that $X$ is a cardinal $\lambda$. $\dot{g}$ is forced to be a countable cofinal subset such that $\bigcup \dot{g} = \lambda$. For each such $\delta$, $\mathrm{Nm}(Z,F) \force \sup_{n} (\sup(\dot{g}_{n} \cap \delta)) = \delta$, as desired. 

 For (3), we refer to~\cite{tsukuura}. Finally, (4) follows as a corollary of (3). 
\end{proof}

By (3) of Lemma~\ref{lem:namba1}, every $\mathrm{Nm}(Z,F)$ can be semiproper. We now recall the notions of semiproperness and semistationary subsets. For a set $W \supseteq \omega_1$, a semistationary subset is a set $S \subseteq [W]^{\omega}$ such that $S^{\mathbf{cl}} = \{x \in [W]^{\omega} \mid \exists y \in S(y \subseteq x \land y \cap \omega_1 = x \cap \omega_1)$ is stationary in $[W]^{\omega}$. For a poset $P$, we say that $P$ is semiproper if and only if the set $\{M \in [\mathcal{H}_{\theta}]^{\omega} \mid \forall p\in M \cap P\exists q \leq p(q \force M[\dot{G}] \cap \omega_1 = M \cap \omega_1)\}$ contains a club.

\begin{lem}[Shelah~\cite{MR1623206}]
 The following are equivalent:
\begin{enumerate}
 \item $P$ is semiproper.
 \item $P$ preserves the semistationarity of any set. 
\end{enumerate}
\end{lem}

The significance of Namba forcing is as follows. (1) of Theorem~\ref{namba:shelah} is not used in this paper, but we introduce it for completeness.

\begin{thm}\label{namba:shelah}\mbox{}
 \begin{enumerate}
  \item (Shelah~\cite{MR1623206} for (a) $\leftrightarrow$ (b); Doebler--Schindler~\cite{MR2576698} for (b) $\rightarrow$ (c)) The following are equivalent:
	\begin{enumerate}
	 \item $(\dagger)$ holds, that is, every $\omega_1$-stationary preserving poset is semiproper. 
	 \item $\mathrm{Nm}(\kappa)$ is semiproper for all regular $\kappa \geq \aleph_2$. 
	 \item $\mathrm{SSR}$ holds, that is, for every $\lambda\geq \aleph_2$ and a semistationary subset $S \subseteq [\lambda]^{\omega}$, there exists an $x \in \mathcal{P}_{\aleph_2}\lambda$ such that $S \cap [x]^{\omega}$ is semistationary.
	\end{enumerate}
  \item (Shelah~\cite{MR1623206}) For all regular $\kappa$, the following are equivalent:
	\begin{enumerate}
	 \item $\mathrm{Nm}(\kappa)$ is semiproper. 
	 \item There exists a semiproper poset that forces $\dot{\mathrm{cf}}(\kappa) = \omega$. 
	\end{enumerate}
 \end{enumerate}
\end{thm}

\begin{thm}[Tsukuura~\cite{tsukuura}]\label{thm:locsp}
Suppose that $\mu^{\omega} < \kappa$ for all $\mu < \kappa$ or that $\kappa = \aleph_2$ and $\aleph_2 \leq \kappa \leq \lambda = \lambda^{<\kappa}$ are regular cardinals. For a semistationary subset $S \subseteq [\lambda]^{\omega}$, the following are equivalent:
\begin{enumerate}
 \item $\mathrm{Nm}(\kappa,\lambda)$ forces $S$ to be semistationary.
 \item There exists an $a \in \mathcal{P}_{\kappa}\lambda$ such that $S \cap [a]^{\omega}$ is semistationary.
 \item For every $\kappa$-complete fine filter $F$ over $\mathcal{P}_{\kappa}\lambda$, $\mathrm{Nm}(\kappa,\lambda,F)$ forces $S$ to be semistationary.
\end{enumerate}
Note that the implication from (1) to (2) holds even if $\mu^{\omega} \geq \kappa$ for some $\mu < \kappa$. 
\end{thm} 

We say that $\mathrm{Nm}(\kappa,\lambda,F)$ is \emph{locally semiproper} if, for every semistationary subset $S \subseteq [\lambda]^{\omega}$, $\mathrm{Nm}(\kappa,\lambda,F) \force S$ is semistationary. The local semiproperness of $\mathrm{Nm}(\lambda,F)$ is defined analogously. The following lemma is used at times.

\begin{lem}[Menas~\cite{MR0357121}]\label{menas}
Let $W$ and $\overline{W}$ be sets such that $\omega_1 \subseteq W \subseteq \overline{W}$. 
\begin{enumerate}
 \item If $C \subseteq [W]^{\omega}$ is a club, then $C^{\downarrow \overline{W}} = \{\overline{x} \in [\overline{W}]^{\omega} \mid \overline{x} \cap W \in C\}$ is a club in $[\overline{W}]^{\omega}$.
 \item If $\overline{C} \subseteq [\overline{W}]^{\omega}$ is a club, then $\overline{C}^{\uparrow W} = \{\overline{x} \cap W \mid \overline{x} \in \overline{C}\}$ contains a club in $[W]^{\omega}$. 
\end{enumerate}
\end{lem}

In particular, for a set $S \subseteq [W]^{\omega} \subseteq [\overline{W}]^{\omega}$, $S$ is semistationary in $[W]^{\omega}$ if and only if it is semistationary in $[\overline{W}]^{\omega}$. Therefore, when considering the semistationarity of a set $S$, we do not emphasize its base set. 

For a set $S \subseteq [W]^{\omega}$, we introduce the specific set
$\mathrm{Gal}_{\theta}(Z,F,S,A)$, the set of all $M \in [\mathcal{H}_{\theta}]^{\omega}$ such that $M \cap W \in S$ and $\bigcap\{f^{-1} M \cap \omega_1 \mid f:A \to \omega_1 \in M\} \in F^\ast$. The idea behind this definition comes from~\cite{MR1261218}. Note that the \emph{non}-stationarity of $\mathrm{Gal}_{\theta}(\aleph_2,I_{\aleph_2}, [\aleph_2]^{\omega},\aleph_2)$ is equivalent to the strong Chang's conjecture at $\aleph_2$. The following lemma is straightforward.

\begin{lem}\label{lem:gal1}
Suppose that $F$ is an $\aleph_2$-complete fine filter over $Z$ and $A \in F^{+}$. Then the following hold:
\begin{enumerate}
 \item If $S \subseteq T$ are subsets of $[W]^{\omega}$, then $\mathrm{Gal}_{\theta}(Z,F,S,A) \subseteq \mathrm{Gal}_{\theta}(Z,F,T,A)$. 
 \item For any set $S \subseteq [W]^{\omega}$, we have $(\mathrm{Gal}_{\theta}(Z,F,S,A))^{\downarrow W} \subseteq S$. 
\end{enumerate}
\end{lem}

Namba forcing adds a club subset of $[\mathcal{H}_{\theta}]^{\omega}$ that is disjoint from $\mathrm{Gal}_{\theta}(Z,F,S,A)$.

\begin{lem}\label{lem:gal2}
For a set $\omega_1 \subseteq W$, a stationary subset $S \subseteq [W]^{\omega}$, a $\kappa$-complete fine filter over $Z \subseteq \mathcal{P}(X)$, and $A \in F^{+}$, there exist $p_{A} \in \mathrm{Nm}(Z,F)$ and a $\mathrm{Nm}(Z,F)$-name $\dot{C}$ for a club subset of $[\dot{\mathcal{H}}_{\theta}]^{\omega}$ such that
\[
p_{A} \force \dot{C} \cap (\mathrm{Gal}_{\theta}(Z,F,S,A))^{\mathbf{cl}} = \emptyset.
\]
In particular, $\mathrm{Gal}_{\theta}(Z,F,S,A)$ is forced to be a \emph{non}-semistationary subset of $S$ by $p_A$, even if $\mathrm{Gal}_{\theta}(Z,F,S,A)$ is stationary. 
\end{lem}

\begin{proof}
Fix a condition $p_{A} \in \mathrm{Nm}(Z,F)$ such that $\mathrm{tr}(p_A) = \emptyset$ and $\mathrm{Suc}_{p_{A}}(s) \subseteq A$ for all $s \in p_{A}$. Let $\dot{\mathcal{A}}$ be a $\mathrm{Nm}(Z,F)$-name for an expansion $\langle \dot{\mathcal{H}}_{\theta},\in,Z,F,\dot{G}\rangle$. Suppose, for contradiction, that $p_{A}$ does not force 
\[
\{M \prec \dot{\mathcal{A}} \mid |M| = \aleph_0\} \cap (\mathrm{Gal}_{\theta}(Z,F,S,A))^{\mathbf{cl}} = \emptyset.
\]
Then there exist $q \leq p_{A}$, $M \in [\mathcal{H}_{\theta}]^{\omega}$, and a $\mathrm{Nm}(Z,F)$-name $\dot{M}$ such that:
\begin{itemize}
 \item $q \force M \subseteq \dot{M} \prec \dot{\mathcal{A}}$.
 \item $M \cap W \in S$. 
 \item $M \cap \omega_1 = \dot{M} \cap \omega_1$. 
 \item $\bigcap\{F^{-1} M \cap \omega_1 \mid F:Z \to \omega_1 \in M\} \cap A \in F^\ast$. 
\end{itemize}

For each function $F:Z \to \omega_1 \in M$, we have
\[
q \force F(\dot{x}) \in M[\dot{G}] \cap \omega_1 = M \cap \omega_1,
\]
where $\dot{x}$ is the $|\mathrm{tr}(q)|$-th element of $\bigcup\{\mathrm{tr}(p) \mid p \in \dot{G}\}$. By a similar argument, we obtain
\[
\mathrm{Suc}_{q}(\mathrm{tr}(q)) \subseteq \bigcap\{F^{-1} M \cap \omega_1 \mid F:Z \to \omega_1 \in M\} \cap A \in F^{+}.
\]
This contradicts our assumption. 
\end{proof}

Note that we can also show the inverse direction. 

\begin{lem}\label{lem:gal3}
 For a stationary subset $S \subseteq [W]^{\omega}$, the following are equivalent:
 \begin{enumerate}
  \item $\mathrm{Gal}_{\theta}(Z,F,S,A)$ is stationary.
  \item There exist a stationary subset $T \subseteq S$ and $p_{A} \in \mathrm{Nm}(Z,F)$ such that $p_A \force T$ is \emph{non}-semistationary. 
 \end{enumerate}
\end{lem}
\begin{proof}
 The forward direction follows from Lemmas~\ref{menas},~\ref{lem:gal1}, and~\ref{lem:gal2}. The inverse direction is straightforward. 
\end{proof}

We now recall Galvin games. For $A \in F^{+}$, the Galvin game $\Game_{\mathrm{Gal}}(F,A)$ is a game of length $\omega$ with the following rules:
\begin{center}
 \begin{tabular}{|c||c|c|c|c|c|c|}
  \hline
  Player I & $F_{0}:A \to \omega_1$ & & $\cdots$ & $F_{i}:A \to \omega_1$ & & $\cdots$ \\ \hline
  Player II &  & $\xi_0 < \omega_1$ & $\cdots$ &  & $\xi_i < \omega_1$ & $\cdots$ 
			  \\ \hline
 \end{tabular}
\end{center}
 Player II wins if $\bigcap_{i < \omega} F_{i}^{-1}\sup_{i}\xi_i \in F^{+}$. The connection between Galvin games and sets $\mathrm{Gal}_{\theta}(Z,F,S,A)$ is as follows. 

\begin{lem}\label{lem:semiproperchar}
 For an $\aleph_2$-complete fine filter $F$ over $Z$, if $\omega_1 \subseteq Z$, then the following are equivalent:
 \begin{enumerate}
  \item $\mathrm{Nm}(Z,F)$ is semiproper.
  \item $\mathrm{Gal}_{\theta}(Z,F,[Z \cup {^{A}}\omega_1]^{\omega},A)$ is \emph{non}-stationary for all $A \in F^{+}$ and sufficiently large regular $\theta$.
  \item $\{M \in [Z\cup {^{A}}\omega_1]^{\omega} \mid \bigcap\{f^{-1}M \cap \omega_1 \mid f:A \to \omega_1 \in M\} \in F^\ast\}$ is \emph{non}-stationary for all $A \in F^{+}$. 
  \item Player II has a winning strategy for $\Game_{\mathrm{Gal}}(F,A)$ for all $A \in F^{+}$. 
\end{enumerate}
\end{lem}
\begin{proof}
 By Lemmas~\ref{lem:gal3} and~\ref{menas}, (1) $\leftrightarrow$ (2) $\to$ (3) follows easily. The implication (3) $\to$ (2) follows from Lemma~\ref{menas}. For (1) $\leftrightarrow$ (4), we refer to~\cite{tsukuura}.
\end{proof}

We now present two applications of this lemma. The first is used to prove Theorem~\ref{maintheorem:model}. Recall that a poset $P$ is $\mu$-strategically closed if and only if Player II has a winning strategy for the game $\Game_{\mu}(P)$. The game $\Game_{\mu}(P)$ is a two-player game designed to construct a descending sequence $\langle p_{\alpha} \mid \alpha < \mu \rangle$. Player I chooses $p_\alpha$ if $\alpha$ is an odd ordinal, while Player II chooses $p_\alpha$ if $\alpha$ is an even ordinal (including limit ordinals). If, after the game, the sequence $\langle p_{\alpha} \mid \alpha < \mu \rangle$ is successfully constructed, Player II wins.

\begin{lem}\label{lem:formodel}
For an $\aleph_2$-complete fine filter $F$ over $Z$, 
\begin{enumerate}
 \item If $\langle F^{+},\subseteq \rangle$ contains an $\aleph_1$-closed dense subset, then $\mathrm{Nm}(Z,F)$ is semiproper. 
 \item If $\mathrm{Nm}(Z,F)$ is semiproper, then for every $|Z|+1$-strategically closed poset $P$, $P$ forces $\dot{\mathrm{Nm}}(Z,F)$ to be semiproper. 
 \item Suppose that for all $A \in F^{+}$, there exists an $\aleph_2$-complete fine ultrafilter $U$ over $Z$ such that $F\cup \{A\} \subseteq U$. Then $\mathrm{Nm}(Z,F)$ is semiproper. 
\end{enumerate}
\end{lem}
\begin{proof}
 For (1), let $D$ be an $\aleph_1$-closed dense subset. We describe Player II's strategy: inductively, after Player I chooses $F_i$, Player II selects $\xi_i$ and $A_i \in D$ such that $A_i \subseteq (\bigcap_{j < i}F_{j}^{-1}\xi_j) \cap F_{i}^{-1}\xi_i$. This is possible since $F$ is $\aleph_2$-complete. After the game, since $D$ is $\aleph_1$-closed, $\bigcap_{i}A_i \in F^{+}$ and thus $\bigcap_{i}F_i^{-1}\sup_{i}\xi_i \in F^{+}$. 

 For (2), note that $P$ does not change $\mathcal{P}(Z)$ or $Z$. Since every play of the Galvin game is in the ground model, the semistationarity of $\mathrm{Nm}(Z,F)$ is preserved by $P$. 

 For (3), fix $A \in F^{+}$. By assumption, we can choose an $\aleph_1$-closed fine ultrafilter $U$ such that $F \cup \{A\} \subseteq U$. By (1), Player II has a winning strategy for $\Game_{\mathrm{Gal}}(U,A)$, which is also a winning strategy for $\Game_{\mathrm{Gal}}(F,A)$.
\end{proof}

The following result is used in Figure~\ref{fig:hierarchy}. 

\begin{thm}
Suppose that $\mu^{\omega} < \kappa$ for all $\mu < \kappa$ or that $\kappa = \aleph_2$. If $\mathrm{Nm}(\kappa,2^{|Z|})$ is locally semiproper, then $\mathrm{Nm}(Z,F)$ is semiproper for all $\kappa$-complete fine filters over $Z$.
\end{thm}
\begin{proof}
 Suppose otherwise. Then there exists a $\kappa$-complete fine filter $F$ over $Z$ such that $\mathrm{Nm}(Z,F)$ is \emph{not} semiproper. Then there exists an $A \in F^{+}$ such that the set in Lemma~\ref{lem:semiproperchar} is stationary. Let $S\subseteq [Z\cup {^{A}}\omega_1]^{\omega}$ be such a stationary set. 

Since $\mathrm{Nm}(\kappa,2^{|Z|})$ is semiproper, by Theorem~\ref{thm:locsp}, there exists an $x \in \mathcal{P}_{\kappa}(Z\cup {^{A}}\omega_1)$ such that $S \cap [x]^{\omega}$ is semistationary. By (3) of Lemma~\ref{lem:namba1}, $\mathrm{Nm}(Z,F)$ forces $S \cap [x]^{\omega}$ to be semistationary. By Lemma~\ref{menas}, so is $S$, which contradicts Lemma~\ref{lem:gal2}. 
\end{proof}

Finally, for later purpose, we introduce the following result. 

\begin{thm}[Todor\v{c}evi\'{c}~\cite{MR0657114}]\label{thm:todch}
 If $\mathrm{Nm}(\aleph_2)$ is semiproper, then $2^{\aleph_0} \leq \aleph_2$. 
\end{thm}

\subsection{Minimal walk methods}\label{subsec:walks}

In this section, we review the \emph{minimal walk methods} needed for our proof. 
For a cardinal $\mu$, a $C$-sequence is a sequence $\langle C_{\alpha} \mid \alpha < \mu \rangle$ with the following properties:
\begin{itemize}
 \item $C_{\alpha+1} = \{\alpha\}$ for all $\alpha < \mu$.
 \item $C_{\alpha} \subseteq \alpha$ is a club for all $\alpha \in \mathrm{Lim} \cap \mu$. 
\end{itemize} 
For a fixed $C$-sequence $\overline{C} = \langle C_{\alpha} \mid \alpha < \mu \rangle$ and ordinals $\alpha \leq \beta < \mu$, a walk from $\beta$ to $\alpha$ is a finite sequence $\beta_0 > \cdots > \beta_{n}$ defined as follows:
\begin{itemize}
 \item $\beta_0 = \beta$.
 \item $\beta_{n+1} = \min (C_{\beta_{n}} \setminus \alpha)$.
 \item $\beta_n = \alpha$. 
\end{itemize}
It is easy to see that a walk always exists and is unique. Let $\rho_2(\alpha,\beta)$ denote the length of the walk. That is, the function $\rho_2:[\mu]^{2} \to \omega$ is defined by $\rho_{2}(\alpha,\beta) = \rho_2(\alpha,\beta_{1}) + 1$ and $\rho_2(\alpha,\alpha) = 0$. 

In the study of walks, the structure of the lower part is important. The \emph{full code} $\rho_0:[\mu]^{2} \to \mu$ is defined by $\rho_0(\alpha,\beta) = \rho_0(\alpha,\beta_1) {^\frown} \langle \mathrm{ot}(C_{\beta_{0}} \cap \alpha)\rangle$ and $\rho_0(\alpha,\alpha) = \langle \rangle$. The function $\rho_0$ \emph{encodes} the behavior of walks as follows:

\begin{lem}
For ordinals $\gamma < \alpha < \beta < \mu$, the following are equivalent:
\begin{enumerate}
 \item $\rho_0(\gamma,\beta) = \rho_0(\gamma,\alpha) {^{\frown}} \rho_0(\alpha,\beta)$.
 \item If a finite sequence $\beta = \beta_0 > \cdots > \beta_{n} = \gamma$ is a walk from $\beta$ to $\gamma$, then $\beta_{i} = \alpha$ for some $i$. Moreover, we can compute such an $i$ as $|\rho_0(\alpha,\beta)|$. 
\end{enumerate}
\end{lem}
\begin{proof}
Straightforward.
\end{proof}

Define $\lambda(\alpha,\beta) = \max\{\sup(C_{\beta_{i}} \cap \alpha) \mid i < \rho_2(\alpha,\beta)\}$. For later purposes, we introduce basic properties of $\lambda(\alpha,\beta)$.

\begin{lem}\label{walk2}
For $\alpha \leq \beta < \mu$, the following hold:
\begin{enumerate}
 \item $\lambda(\alpha,\beta) \leq \alpha$. 
 \item If $0 < \rho_2(\alpha,\beta)$, then $\lambda(\beta_{\rho_2(\alpha,\beta)-1},\beta) < \alpha$. Here, $\beta_{0} > \cdots > \beta_{n}$ is a walk from $\beta$ to $\alpha$. 
 \item For all ordinals $\lambda(\alpha,\beta) < \gamma < \alpha \leq \beta$, we have $\rho_0(\gamma,\beta) = \rho_0(\alpha,\beta){^\frown} \rho_0(\gamma,\alpha)$. 
\end{enumerate}
\end{lem}
\begin{proof}
Statements (1) and (2) are straightforward. We now prove (3). 

Let $\beta_{0} > \cdots > \beta_{n}$ and $\beta_{0}' > \cdots > \beta_{m}$ be walks from $\beta$ to $\alpha$ and $\gamma$, respectively. We show by induction on $i$ that $\beta_{i} = \beta_{i}'$ for all $i \leq n$. Clearly, we have $\beta_{0} = \beta_{0}'$. Suppose that $\beta_{i} = \beta_{i}'$ for all $i < n$. Since $\sup (C_{\beta_i} \cap \alpha) \leq \lambda(\alpha,\beta) < \gamma$, 
it follows that
\[
\beta'_{i+1} = \min (C_{\beta'_i} \setminus \gamma) = \min (C_{\beta_i} \setminus \gamma) = \min (C_{\beta_i} \setminus \alpha) = \beta_{i+1},
\]
as desired.
\end{proof}

If a $C$-sequence $\overline{C}$ satisfies the following properties, we call $\overline{C}$ a $\square(\mu,{<}\kappa)$-sequence. 
\begin{enumerate}
 \item For all $\alpha \in \mathrm{Lim}\cap \mu$, 
 \[
 |\{C_\beta \cap \alpha \mid \alpha \leq \beta \land \alpha \in \mathrm{Lim}(C_{\beta}) \}| < \kappa.
 \]
 \item $\overline{C}$ is \emph{non-trivial}, meaning that there is no club $C \subseteq \mu$ such that, for every $\alpha < \mu$, there exists $\beta \ge \alpha$ with $C \cap \alpha \subseteq C_{\beta}$. 
\end{enumerate}
We write $\square(\mu,{<}\kappa)$ to denote the existence of a $\square(\mu,{<}\kappa)$-sequence. We write $\square(\mu)$ for $\square(\mu,{<}2)$. 

The principle $\square(\mu)$ may fail, and if it fails at two consecutive cardinals, then there exists an inner model with a large cardinal property. The following theorems will be referenced in Section~\ref{sec:maintheorem}. 

\begin{thm}[Steel]\label{sqinnermodel1}
If $\kappa \geq 2^{\aleph_0}+\aleph_2$ is a regular cardinal and both $\square(\kappa)$ and $\square_{\kappa}$\footnote{$\square_{\kappa}$ asserts the existence of a $\square(\kappa^{+})$-sequence $\overline{C}$ such that $\mathrm{ot}(C_{\alpha}) \leq \kappa$ for all $\alpha \in \mathrm{Lim} \cap \kappa^{+}$.} fail, then there exists an inner model with infinitely many Woodin cardinals. 
\end{thm}

\begin{thm}[Jensen--Schimmerling--Schindler--Steel~\cite{mice}]\label{sqinnermodel2}
If $\kappa \geq \aleph_3$ is a regular cardinal such that $\kappa^{\aleph_0} = \kappa$ for all $\mu < \kappa$, and both $\square(\kappa)$ and $\square_{\kappa}$ fail, then there exists an inner model with a proper class of strong cardinals and a proper class of Woodin cardinals. 
\end{thm}

For ordinals $\alpha \leq \beta < \mu$, define $\overline{\lambda}(\alpha,\beta) = \max (\{\lambda(\alpha,\beta), \lambda(\beta_{\rho_{2}(\alpha,\beta) - 1}, \beta) \} \cap \alpha)$.
By (2) of Lemma~\ref{walk2}, the inequality $\overline{\lambda}(\alpha,\beta) < \alpha$ always holds, even if $\lambda(\alpha,\beta) = \alpha$. The ordinal $\overline{\lambda}(\alpha,\beta)$ is one of our main tools for studying $C$-sequences.

\begin{lem}\label{walk4}
 For ordinals $\alpha < \beta < \mu$, let $m = \rho_2(\alpha,\beta)$ and $\overline{\lambda} = \overline{\lambda}(\alpha,\beta)$. The following hold:
 \begin{enumerate}
  \item If $\overline{\lambda} = \lambda(\alpha,\beta)$, then 
\[
C_{\alpha}\setminus (\overline{\lambda} + 1) = \{\gamma < \alpha \mid \rho_2(\gamma,\beta) = m + 1\}\setminus (\overline{\lambda} + 1).  
\]
  \item If $\overline{\lambda} < \lambda(\alpha,\beta)$, then 
\[
(C_{\beta_{m-1}}\cap \alpha) \setminus (\overline{\lambda}+1) = \{\gamma < \alpha \mid \rho_2(\gamma,\beta) = m \} \setminus (\overline{\lambda}+1).
\]
 \end{enumerate}
Here, $\beta_0 > \cdots > \beta_{m}$ is a walk from $\beta$ to $\alpha$. 
\end{lem}

\begin{proof}
 First, we verify (1). Since $\overline{\lambda}(\alpha,\beta) = \lambda(\alpha,\beta)$, we have $\lambda(\alpha,\beta) < \alpha$. For every $\gamma \in [\lambda+1,\alpha)$, by (3) of Lemma~\ref{walk2}, we obtain $\rho_0(\gamma,\beta) = \rho_0(\alpha,\beta){^{\frown}}\rho_0(\alpha,\alpha)$.
It is easy to see that $|\rho_0(\alpha,\alpha)| = 1$ if and only if $\alpha \in C_{\alpha}$, by the definition of walks. Note that $\rho_2(\gamma,\beta) = |\rho_0(\gamma,\beta)| = m + \rho_{2}(\gamma,\alpha)$.
Thus, $\rho_2(\gamma,\beta) = m + 1$ is equivalent to $\gamma \in C_{\alpha}$, as desired.

Next, we prove (2). By (2) of Lemma~\ref{walk2}, we have $\overline{\lambda}(\alpha,\beta) = {\lambda}(\beta_{m-1},\beta)$. For every $\gamma \in [\overline{\lambda}+1,\alpha)$, by (3) of Lemma~\ref{walk2},
$\rho_0(\gamma,\beta) = \rho_0(\beta_{m-1},\beta){^{\frown}}\rho_0(\gamma,\beta_{m-1})$.
Similarly, $\gamma \in C_{\beta_{m-1}}$ if and only if $\rho_2(\gamma,\beta) = (m-1)+1 = m$, as desired.
\end{proof}
The proof of Lemma~\ref{walk4} includes the following lemma.
\begin{lem}\label{walk5}For ordinals $\alpha < \beta < \mu$, let $m = \rho_2(\alpha,\beta)$ and $\overline{\lambda} = \overline{\lambda}(\alpha,\beta)$. The following holds:
\begin{enumerate}
 \item If $\overline{\lambda} = \lambda(\alpha,\beta)$ then 
\[
 \rho_2(\gamma,\beta) = \rho_2(\gamma,\alpha) + m\text{ for all }\gamma \in [\overline{\lambda},\alpha).
\]
 \item If $\overline{\lambda} < \lambda(\alpha,\beta)$ then 
\[
 \rho_2(\gamma,\beta) = \rho_2(\gamma,\beta_{m-1}) + m - 1\text{ for all }\gamma \in [\overline{\lambda},\alpha).
\]
\end{enumerate}
Here, $\beta_0 > \cdots > \beta_{m}$ is a walk from $\beta$ to $\alpha$. 
\end{lem}

Lastly, we conclude this section with the following lemma.

\begin{lem}[Sakai--Veli\u{c}kovi\'{c}~\cite{MR3284479} for $\nu = 2$; Torres-P\'erez--Wu~\cite{torresperezandwu} for $\nu = \aleph_1$]\label{lem:semistationarysubset}
Suppose that $\aleph_2 \leq \mu$ is a regular cardinal and that $\overline{C}$ is a $\square(\mu,{<}\nu)$-sequence for some (possibly finite) cardinal $\nu \leq \aleph_1$. 

Let $S_0 \subseteq [\mu]^{\omega}$ be the set of all countable $x \subseteq \mu$ such that, for all $\beta \geq \delta = \sup (x \cap \mu)$ with $\delta \in \mathrm{Lim}(C_{\beta})$, the following conditions hold:
\begin{itemize}
	\item $x \cap C_{\beta}$ is bounded in $\delta$.
	\item For all sufficiently large $\gamma \in C_{\beta} \cap \delta$, we have $\mathrm{cf}(\min ((x \cap \mu) \setminus \gamma)) = \omega_1$. 
\end{itemize}
Then $S_0$ is a stationary subset. 
\end{lem}

In the proof of Lemma~\ref{lem:semistationarysubset}, Sakai--Veli\u{c}kovi\'{c} used a game originally introduced by Veli\u{c}kovi\'{c}~\cite{MR1174395}. Here, we consider a variation of this game. Let $\mathcal{A} = \langle \mathcal{H}_{\theta},\in,... \rangle$ be an arbitrary countable expansion. Define a game $\Game(\kappa,\mu,\mathcal{A})$ of length $\omega$ with the following rules.

\begin{center}
  \begin{tabular}{|c||ccc|ccc|c|ccc|c|}
   \hline
  Player I & $\alpha_0$ & & $\gamma_0$, $x_{0}$ & $\alpha_1$ & & $\gamma_1$, $x_1$ & $\cdots$ & $\alpha_n$ & & $\gamma_n$, $x_n$ & $\cdots$\\
   \hline
  Player II &  & $y_0$ &  &  &$y_1$ & & $\cdots$ &  & $y_n$&  & $\cdots$\\
   \hline
  \end{tabular}
\end{center}

On the $n$-th turn, Player I first chooses an ordinal $\alpha_n < \mu$. Then, Player II chooses a set of ordinals $y_n \in \mathcal{P}_{\kappa}\mu$. Finally, Player I selects $\gamma_n \in E^\mu_{\omega_1} \setminus ((\alpha_n + 1) \cup \sup y_{n})$ and $x_n \in \mathcal{P}_{\kappa}\mu$ such that $y_n \subseteq x_n$. 

Player I wins if 
\[
\mathrm{Sk}_{\mathcal{A}}(\{\gamma_n,x_n \mid n < \omega\}) \cap [\alpha_m,\gamma_m) = \emptyset
\]
for all $m < \omega$. Otherwise, Player II wins. Since $\Game(\kappa,\mu,\mathcal{A})$ is an open game, one of the players has a winning strategy. 

\begin{lem}\label{lem:game}
For every pair of regular cardinals $\aleph_1 \leq \kappa \leq \mu$ and any countable expansion $\mathcal{A} = \langle \mathcal{H}_{\theta},\in,...\rangle$, Player I has a winning strategy for $\Game(\kappa,\mu,\mathcal{A})$. 
\end{lem}

\begin{proof}
Suppose otherwise. Since $\Game(\kappa,\mu,\mathcal{A})$ is an open game, Player II must have a winning strategy $\tau$. Let $\langle M_n \mid n < \omega \rangle$ be an $\in$-chain such that, for every $n < \omega$, 
\begin{itemize}
 \item $\tau \in M_n \prec  \mathcal{A}$,
 \item $M_{n} \cap \kappa < \kappa$, and 
 \item $|M_n| < \kappa$. 
\end{itemize}
For each $n < \omega$, define
\begin{itemize}
 \item $\gamma_n = \sup M_n \cap \mu$.  
 \item $x_{n} = M_n \cap \mu$. 
\end{itemize}
Let $M = \mathrm{Sk}_{\mathcal{A}}(\{\gamma_n,x_n \mid n < \omega\})$. Define 
$\alpha_n = \min (M_n \setminus (\sup (M \cap \gamma_n))) < \gamma_n$. Since $|M_n| < \kappa$, we have $x_n \in \mathcal{P}_{\kappa}\mu$. 

We claim that there exists a play in which Player I moves $\langle \alpha_n,\gamma_n,x_n \mid n < \omega \rangle$ and Player II plays according to $\tau$. We prove this by induction on $n$. Suppose Player I moves $\alpha_n$. Then Player II chooses $y_n = \tau(\alpha_0,\gamma_0,x_0,\cdots, \alpha_{n-1},\gamma_{n-1},x_{n-1},\alpha_n)$. By definition, $\alpha_0,\gamma_0,x_0,\cdots, \alpha_{n-1},\gamma_{n-1},x_{n-1},\alpha_n,\tau \in M_n$, so $y_n \in M_n$. Since $|y_{n}| < \kappa$ and $M_n \cap \kappa <\kappa$, we have $y_n \subseteq M_n\cap \mu = x_n$ and $\sup y_n < \gamma_n$. On the other hand, by the definition of $\alpha_n$, there is no element of $M$ between $\alpha_n$ and $\gamma_n$, which contradicts Player II's winning strategy. 
\end{proof}

\begin{lem}\label{lem:nontrivial}
If $\overline{C}$ is \emph{non-trivial}, then for every club $C \subseteq \mu$ and stationary $S \subseteq \mu$, there exists $\delta \in \mathrm{Lim}(C) \cap S$ such that $(C \cap \delta) \setminus C_{\beta}$ is unbounded for all $\beta \geq \delta$. 
\end{lem}

\begin{proof}
Suppose otherwise. That is, for every $\delta \in \mathrm{Lim}(C) \cap S$, the set $(C \cap \delta) \setminus C_{\beta(\delta)}$ is bounded in $\delta$ for some $\beta(\delta) \geq \delta$. Define $\eta_{\delta} < \delta$ as a bound for $(C \cap \delta) \setminus C_{\beta(\delta)}$. By Fodor's lemma, we can choose $\eta$ and a stationary subset $T \subseteq \mathrm{Lim}(C) \cap S$ such that $\eta_{\delta} = \eta$ for all $\delta \in T$. 

Let $D = C \setminus \eta$. The set $D$ is a club such that, for all $\alpha$, $D \cap \alpha \subseteq C_{\beta(\min(T \setminus \alpha))}.$
This contradicts the assumption that $\overline{C}$ is non-trivial.
\end{proof}
Torres-P\'erez--Wu proved the following lemma to establish Lemma~\ref{lem:semistationarysubset}. 

\begin{lem}[Torres-P\'erez--Wu~\cite{torresperezandwu}]\label{lem:nontrivial2}
 If $\overline{C}$ is a $\square(\mu,{<}\nu)$-sequence for some (possibly finite) cardinal $\nu \leq \aleph_1$, then for every club $C \subseteq \mu$ and stationary set $S \subseteq E^\mu_{\omega}$, there exists $\delta \in \mathrm{Lim}(C) \cap S$ such that $(C \cap \delta) \setminus (\bigcup_{\delta\in\mathrm{Lim}(C_{\beta})} C_{\beta})$ is unbounded.
\end{lem}

The following lemma is used in Section~\ref{sec:maintheorem}. 

\begin{lem}\label{lem:wqkillstationary}
 Suppose that $\kappa \leq \mu$ are regular cardinals and that $\overline{C}$ is a $\square(\mu,{<}\nu)$-sequence for some (possibly finite) cardinal $\nu \leq \aleph_1$. Let $S_1 \subseteq [\mathcal{H}_{\theta}]^{\omega}$ be the set of all countable $M \prec \mathcal{H}_{\theta}$ such that, for all $\beta \geq \delta = \sup (M \cap \mu)$ with $\delta \in \mathrm{Lim}(C_{\beta})$, the following hold:
\begin{itemize}
	\item $M \cap C_{\beta}$ is bounded in $\delta$.
	\item For all sufficiently large $\gamma \in C_{\beta} \cap \delta$, $\mathrm{cf}(\min ((M \cap \mu) \setminus \gamma)) = \omega_1$. 
	\item For each $\xi < \delta$, there exists $x \in \mathcal{P}_{\kappa}\mu \cap M$ such that $(x \cap C_{\beta} \cap \delta) \setminus \xi \not= \emptyset.$
\end{itemize}
Then $S_1$ is a stationary subset. 
\end{lem}

\begin{proof}
 Let $\mathcal{A} = \langle \mathcal{H}_{\theta},\in,... \rangle$ be an arbitrary countable expansion. By Lemma~\ref{lem:game}, Player I has a winning strategy for $\Game(\kappa,\mu,\mathcal{A})$. Let $\langle M_\xi \mid \xi < \mu \rangle$ be an $\in$-increasing cofinal elementary chain of $\langle \mathcal{A},\tau\rangle$ such that $M_\xi \cap \mu < \mu$ for all $\xi < \mu$. 

Note that $C = \{\xi \mid M_\xi \cap \mu = \xi, \ \xi < \mu \}$ is a club. By Lemma~\ref{lem:nontrivial2}, there exists $\delta \in \mathrm{Lim}(C) \cap E_{\omega}^{\mu}$ such that, for all $\beta \geq \delta$, 
$(C \cap \delta) \setminus \bigcup_{\delta \in \mathrm{C}_{\beta}}C_{\beta}$ is unbounded in $\delta$. 

Since $\overline{C}$ is a $\square(\kappa,{<}\nu)$-sequence and $\nu \leq \aleph_1$, we can enumerate 
\[
\{\beta \mid \delta \in \mathrm{Lim}(C_{\beta})\} = \{\beta_{i} \mid i < \omega\}.
\]
The appearance of each $\beta_{i}$ may be repeated if $\nu \leq \aleph_0$. 

Pick a sequence $\{\delta_n \mid n < \omega \} \subseteq (C\cap \delta) \setminus \bigcup_{\delta \in \mathrm{Lim}(\mathrm{C}_{\beta})}C_{\beta}$ such that $\sup_{n} \delta_n = \delta$. For each $n$, let $y_n = \{\max(C_{\beta_i} \cap \delta_n) \mid i \leq n\}$.
Since $\overline{C}$ is a $\square(\kappa,{<}\mu)$-sequence, we have $y_{n} \in \mathcal{P}_{\kappa}\mu$. Consider $\langle y_{n} \mid n < \omega \rangle$ as Player II's moves. Let $\langle \alpha_n,\gamma_n,x_n \mid n < \omega \rangle$ be Player I's moves according to $\tau$. Define $M = \mathrm{Sk}_{\mathcal{A}}(\{\gamma_n,x_n \mid n < \omega\})$.
Since $y_{n} \subseteq \delta_n = M_{\delta_n} \cap \mu$ is finite, we obtain $y_{n} \in M_{\delta}$. Therefore, we can define $\alpha_{n+1}$ in $M_{n}$ using $\tau$, so that $\alpha_{n+1} < \delta_{n}$. By the rules of $\Game(\kappa,\mu,\mathcal{A})$, we have $\max (C_{\beta_{i}} \cap \delta_{n}) < \gamma_{n+1}$ for each $i \geq n$.

Lastly, we check that $M \in S_1$. For each $\beta \geq \delta$ with $\delta \in \mathrm{Lim}(C_{\beta})$, we have 
\[
C_\beta \cap \delta = C_{\beta_{i}} \cap \delta
\]
for some $i < \omega$. By the choice of $y_{i}$, for each $n \geq i$, 
\[
C_{\beta} \cap [\delta_{n},\delta_{n+1}) = C_{\beta_{i}} \cap [\delta_n ,\delta_{n+1}) \subseteq M \cap [\alpha_{n+1},\gamma_{n+1}).
\]
Since Player I wins, we have $M \cap [\alpha_{n+1},\gamma_{n+1}) = \emptyset$, which implies $C_{\beta} \cap [\delta_{n},\delta_{n+1}) = \emptyset$. In particular, $M \cap C_{\beta}$ is bounded by some ordinal $\xi < \delta$. 

For each $\zeta \geq \xi$, we can choose $n > i$ such that, letting $\gamma = \max(C_{\beta} \cap \delta_{n})$, we have
\begin{itemize}
 \item $\gamma = \max(C_{\beta_{i}} \cap \delta_n) \in x_{n}\in M\cap \mathcal{P}_{\kappa}\mu$.
 \item $\delta_{n} \geq \zeta$.
\end{itemize}
Then, $\min (M \setminus \gamma) = \gamma_{n}$, and its cofinality is $\omega_1$, as desired.
\end{proof}

For Section~\ref{sec:reflections}, we examine another variation of Lemma~\ref{lem:semistationarysubset}.

\begin{lem}\label{lem:killstationary}
 Suppose that $\aleph_2 \leq \kappa \leq \mu$ are regular cardinals and that $\overline{C}$ is nontrivial. Let $S_1 \subseteq [\mathcal{H}_{\theta}]^{\omega}$ be the set of all countable $M \prec \mathcal{H}_{\theta}$ such that: 
       \begin{itemize}
	\item $M \cap C_{\sup M \cap \mu}$ is bounded in $\sup M \cap \mu$. 
	\item For all sufficiently large $\gamma \in C_{\sup M \cap \mu}$, 
	$\mathrm{cf}(\min ((M \cap \mu) \setminus \gamma)) = \omega_1$.
	\item For each $\xi < \sup M \cap\mu$, there exists $a \in \mathcal{P}_{\kappa}\mu \cap M$ such that $(a \cap C_{\sup M \cap \mu}) \setminus \xi \neq \emptyset$.
       \end{itemize}
Then $S_1$ is a stationary subset. 
\end{lem}

\begin{proof}
By Lemma~\ref{lem:nontrivial}, for every countable expansion $\mathcal{A} = \langle \mathcal{H}_{\theta},\in,...\rangle$, we can find $\delta \in E_{\omega}^{\mu}$ and an elementary substructure $M_n \prec \mathcal{A}$ such that: 
\begin{itemize}
 \item $\delta_n = M_n \cap \mu \in \mu \setminus C_{\delta}$.
 \item $M_n \in M_{n+1}$.
 \item $\sup_{n} \delta_{n} = \delta$. 
\end{itemize}
Define $\beta_{n} = \max (C_{\delta} \cap \delta_{n})$. As in the proof of Lemma~\ref{lem:wqkillstationary}, there exists a play of $\Game(\kappa,\mu,\mathcal{A})$ in which Player I wins and Player II chooses $\{\beta_{n}\}$ in each turn. This play also defines $M \prec \mathcal{A}$ such that $M \in S_{1}$, as desired.
\end{proof}

In Sections~\ref{sec:walksquare},~\ref{sec:maintheorem}, and \ref{sec:reflections}, We will show that $S_0$ is \emph{not} semistationary in the extension by $\mathrm{Nm}(\mu)$ or a specific $\mathrm{Nm}(\kappa,\lambda,F)$.

\subsection{Two-Cardinal Combinatorics}\label{subsec:combinatorics}

In this section, for cardinals $\kappa \leq \lambda$, we introduce a two-cardinal tree property $\mathrm{TP}(\kappa,\lambda)$ and a two-cardinal partition property $\mathcal{P}_{\kappa}\lambda \to [I_{\kappa\lambda}^{+}]^{n}$. 

First, we recall the tree property $\mathrm{TP}(\kappa,\lambda)$, originally introduced by Jech~\cite{jechstronglycompact}. 
A \emph{$\mathcal{P}_{\kappa}\lambda$-list} is a sequence $\overline{d} = \langle d_{x} \mid x \in \mathcal{P}_{\kappa}\lambda \rangle$, where $d_{x} \subseteq x$ for all $x \in \mathcal{P}_{\kappa}\lambda$. We say that $\overline{d}$ is \emph{thin} if there exists a club $C \subseteq \mathcal{P}_{\kappa}\lambda$ such that $|\{d_{x} \cap y \mid y \subseteq x\}| < \kappa$ for all $y \in C$. We denote by $\mathrm{Lev}_{y}(\overline{d})$ the set $\{d_{x} \cap y \mid y \subseteq x\}$. A \emph{branch} of $\overline{d}$ is a set $d \subseteq \lambda$ such that, for all $x \in \mathcal{P}_{\kappa}\lambda$, $d \cap x \in \mathrm{Lev}_{x}(\overline{d})$.
The property $\mathrm{TP}(\kappa,\lambda)$ asserts the \emph{nonexistence} of a thin $\mathcal{P}_{\kappa}\lambda$-list with \emph{no} branches. Note that $\mathrm{TP}(\kappa,\kappa)$ is equivalent to the \emph{nonexistence} of a $\kappa$-Aronszajn tree. These tree properties characterize the strong compactness of $\kappa$ via the following theorem, which we will use in Theorem~\ref{maintheorem2}. 

\begin{thm}[Jech~\cite{jechstronglycompact}]\label{jech}
For regular cardinals $\kappa \leq \lambda$, the following are equivalent:
\begin{enumerate}
 \item The logic $\mathcal{L}_{\kappa\kappa}$ satisfies the compactness theorem for any theory of size $\lambda$.
 \item $\kappa$ is inaccessible and $\mathrm{TP}(\kappa,\lambda)$ holds.
\end{enumerate}
\end{thm}

Lastly, we introduce a two-cardinal version of the square-bracket partition property. For cardinals $\kappa \leq \lambda$, the statement $\mathcal{P}_{\kappa}\lambda \not\to [I_{\kappa\lambda}^{+}]^{n}_{\lambda}$ asserts the existence of a function $c:[\mathcal{P}_{\kappa}\lambda]^{n} \to \lambda$
such that $c ``[X]^{n}_{\subset} = \lambda$ for all $X \in I_{\kappa\lambda}^{+}$. Here, $[X]^{n}_{\subset} = \{\{x_{0},...,x_{n-1}\} \subseteq X \mid x_{0} \subset \cdots \subset x_{n-1} \}$.

The following theorems are not used in this paper, but we introduce them to explain our motivations. Theorem~\ref{maintheorem3} aims to generalize Theorem~\ref{todorcevic} in the framework of Theorem~\ref{todo2}. Note that Theorem~\ref{todorcevic} includes Theorem~\ref{todo2} as a special case when $\kappa = \lambda$.

\begin{thm}[Todor\v{c}evi\'{c}~\cite{MR2355670}]\label{todorcevic}
If $\kappa = \nu^{+}$ for some regular $\nu$, then 
$\mathcal{P}_{\kappa}\lambda \not\to [I_{\kappa\lambda}^{+}]^{2}_{\lambda}$
for all regular $\lambda \geq \kappa$. 
\end{thm}

\begin{thm}[Todor\v{c}evi\'{c}~\cite{10.1007/BF02392561}]\label{todo2}
If there exists a stationary subset $S \subseteq \kappa$ such that $S \cap \alpha$ is \emph{non}-stationary for all $\alpha \in \mathrm{Lim}$, then $\kappa \not\to [\kappa]_{\kappa}^{2}$. In particular, if $\kappa = \nu^{+}$ for some regular $\nu$, then $\kappa \not\to [\kappa]^{2}_{\kappa}$.
\end{thm}

For a stationary subset $S \subseteq \lambda$ and regular cardinals $\aleph_1 \leq \kappa \leq \lambda$, the property $\mathrm{Refl}(S,{<}\kappa)$, introduced in ~\cite{MR1838355}, asserts that for every stationary subset $T \subseteq S$, there exists an ordinal $\alpha \in E_{{<}\kappa}^{\lambda}$ with cofinality $\geq \omega_1$ such that $T \cap \alpha$ is stationary in $\alpha$. 

We write $\mathrm{Refl}(S)$ to denote $\mathrm{Refl}(S,{<}\sup S)$. If $T\subseteq S$ witnesses $\lnot\mathrm{Refl}(S,{<}\kappa)$, then we call $T$ a \emph{non}-reflecting stationary subset. Non-reflecting stationary subsets play an important role in walk theory. In this paper, we use these in Sections~\ref{sec:reflections} and~\ref{sec:twocardinalwalks}.

\section{A filter associated with a thin list}\label{sec:filter}
In this section, we construct a $\kappa$-complete fine filter $F_{\overline{d}}$ over $\mathcal{P}_{\kappa}\lambda$ for a thin $\mathcal{P}_{\kappa}\mu$-list. Hayut used the following lemma to define a filter over $\lambda$. 

\begin{lem}[Fichtenholz--Kantorovic for $\kappa = \lambda = \omega$; Hausdorff for $\kappa =\aleph_0 \leq \lambda$; Hayut~\cite{MR3959249} for general $\aleph_0\leq\kappa\leq \lambda$]\label{lemma:hausdorff}
For infinite cardinals $\kappa \leq \lambda = \lambda^{<\kappa}$, there exists a family $\mathcal{X}\subseteq [\lambda]^{\lambda}$ such that $|\mathcal{X}| = 2^{\lambda}$ and, for all $x,y\in \mathcal{P}_{\kappa}\mathcal{X}$ with $x \cap y = \emptyset$, 
\[
|(\bigcap x) \setminus (\bigcup y)| = \lambda.
\]
\end{lem}

\begin{proof}
 See~\cite[Lemma~2]{MR3959249}. 
\end{proof}

The following lemma is a $\mathcal{P}_{\kappa}\lambda$ version of Lemma~\ref{lemma:hausdorff} and plays a central role. 

\begin{lem}\label{lemma:propvaluesets}
 For regular cardinals $\kappa \leq \lambda = \lambda^{<\kappa}$, there exists a family $\langle A_{\alpha} \mid \alpha < 2^{\lambda} \rangle$ of subsets of $\mathcal{P}_{\kappa}\lambda$ such that, for every $x,y \in \mathcal{P}_{\kappa}2^{\lambda}$ with $x \cap y = \emptyset$, 
 \[
 \bigcap_{\alpha \in x} A_{\alpha} \setminus \bigcup_{\alpha \in y} A_{\alpha} \text{ is unbounded in } \mathcal{P}_{\kappa}\lambda.
 \]
\end{lem}

\begin{proof}
For each $X \in [\lambda]^{\lambda}$, define 
\[
g(X) = \{x \in \mathcal{P}_{\kappa}\lambda \mid \sup x \cap X = \sup x\}.
\]
If $|X \mathrel{\triangle} Y| = \lambda$, then for each $z \in \mathcal{P}_{\kappa}(X \setminus Y) \cup \mathcal{P}_{\kappa}(Y\setminus X)$ with $\sup z \not\in z$, we have $z \in g(X) \mathrel{\triangle} g(Y)$. Thus, if we take a family $\mathcal{X} \subseteq [\lambda]^{\lambda}$ such that $|X \mathrel{\triangle} Y| = \lambda$ for all $\{X,Y\} \in [\mathcal{X}]^2$, then we obtain $|\{g(X) \mid X \in \mathcal{X}\}| = |\mathcal{X}|$.

By Lemma~\ref{lemma:hausdorff}, fix a family $\mathcal{X} \subseteq [\lambda]^{\lambda}$ of size $2^{\lambda}$ such that, for all $a,b \in \mathcal{P}_{\kappa}\mathcal{X}$ with $a \cap b = \emptyset$, $|(\bigcap a) \setminus (\bigcup b)| = \lambda$.
For such a pair $\langle a, b\rangle$ and $x' \in \mathcal{P}_{\kappa}\lambda$, we can choose $x \in \mathcal{P}_{\kappa}((\bigcap a) \setminus (\bigcup b))$ such that $\sup x \not\in x$. Since $|(\bigcap a) \setminus (\bigcup b)| = \lambda$, we may assume that $\min x \geq \sup x'$. 

By this definition, $x \cup x' \in \bigcap_{X \in a} g(X) \setminus \bigcup_{Y \in b} g(Y)$. Hence, $\bigcap_{X \in a} g(X) \setminus \bigcup_{Y \in b} g(Y)$ is unbounded in $\mathcal{P}_{\kappa}\lambda$, as required.
\end{proof}

For regular cardinals $\kappa \leq \lambda = \lambda^{<\kappa} \leq \mu \leq 2^{\lambda}$ and a thin $\mathcal{P}_{\kappa}\mu$-list $\overline{d} = \langle d_x \mid x \in \mathcal{P}_{\kappa}\mu \rangle$, we define a $\kappa$-complete filter $F_{\overline{d}}$ over $\mathcal{P}_{\kappa}\lambda$. Let $C$ be a club such that, for all $x \in \mathcal{P}_{\kappa}\lambda$, there exists $c_x \in C$ with $x \subseteq c_x$ and $|\mathrm{Lev}_{c_x}(\overline{d})| < \kappa$.

For $x \in \mathcal{P}_{\kappa}\mu$, define $B_x$ by 
\[
B_x = \bigcup_{d \in \mathrm{Lev}_{c_x}(\overline{d})} \left( \bigcap_{\alpha \in d \cap x} A_{\alpha} \setminus \bigcup_{\alpha \in x \setminus d} A_{\alpha} \right).
\]
Here, $\langle A_{\alpha} \mid \alpha < 2^{\lambda} \rangle$ comes from Lemma~\ref{lemma:propvaluesets}. By Lemma~\ref{lemma:propvaluesets}, $B_x$ is unbounded in $\mathcal{P}_{\kappa}\lambda$. 

\begin{lem}\label{lemma:completeness}
The family 
\[
\{B_x \mid x \in \mathcal{P}_{\kappa}\mu\} \cup \{\{x \in \mathcal{P}_{\kappa}\lambda \mid \xi \in x\} \mid \xi < \lambda\}
\]
generates a $\kappa$-complete filter over $\mathcal{P}_{\kappa}\lambda$.
\end{lem}

\begin{proof}
For every $X \in [\mathcal{P}_{\kappa}\mu]^{<\kappa}$, we can fix $c \in C$ such that $\bigcup_{x \in X} c_x \subseteq c$. We take $d \in \mathrm{Lev}_{c}(\overline{d})$. Then the following hold:
\begin{itemize}
 \item $d \cap c_x \in \mathrm{Lev}_{c_x}(\overline{d})$.
 \item $|\bigcap_{\alpha \in d} A_{\alpha} \setminus \bigcup_{\alpha \in c \setminus d} A_{\alpha}| = \lambda$.
\end{itemize}
Let 
$B' = \bigcap_{\alpha \in d} A_{\alpha} \setminus \bigcup_{\alpha \in c \setminus d} A_{\alpha}$. It suffices to prove that $B' \subseteq \bigcap_{x \in X} B_x$. 

For each $x \in X$, we have 
\[
B' \subseteq \bigcap_{\alpha \in (d \cap c_x) \cap x} A_{\alpha} \setminus \bigcup_{\alpha \in x \setminus (d \cap c_x)} A_{\alpha}.
\]
By $d \cap c_x \in \mathrm{Lev}_{c_x}(\overline{d})$, we conclude that $B' \subseteq B_x$. 

Thus, the proof is completed.
\end{proof}

Let $F_{\overline{d}}$ be the $\kappa$-complete fine filter defined by Lemma~\ref{lemma:completeness}. The importance of $F_{\overline{d}}$ is as follows. 
\begin{lem}\label{lemma:branchext}
If there exists a $\kappa$-complete fine ultrafilter $U$ over $\mathcal{P}_{\kappa}\lambda$ such that $F_{\overline{d}} \subseteq U$, then $\overline{d}$ has a branch.
\end{lem}
\begin{proof}
Define $d \subseteq \mu$ by setting $\alpha \in d$ if and only if $A_{\alpha} \in U$. We claim that $d$ is a branch of $\overline{d}$. For every $x \in \mathcal{P}_{\kappa}\mu$, since $B_x \in U$ and $U$ is a $\kappa$-complete ultrafilter, there exists $e \in \mathrm{Lev}_{c_x}(\overline{d})$ such that 
\[
\bigcap_{\alpha \in e \cap a} A_{\alpha} \setminus \bigcup_{\alpha \in a \setminus e} A_{\alpha} \in U.
\]
Thus, we have:
\begin{itemize}
 \item For all $\alpha \in e \cap x$, $A_{\alpha} \in U$.
 \item For all $\alpha \in x \setminus e$, $A_{\alpha} \not\in U$.
\end{itemize}
By the definition of $d$, it follows that $d \cap x = e \cap x$. Since $e \in \mathrm{Lev}_{c_x}(\overline{d})$, there exists $y \in \mathcal{P}_{\kappa}\mu$ such that $x \subseteq c_a \subseteq y$ and $e = d_{y} \cap c_x$. Therefore, $d \cap x = e \cap x = d_{y} \cap x \in \mathrm{Lev}_x(\overline{d})$, as required.
\end{proof}

Note that $F_{\overline{d}}$ may simply be the bounded ideal. For example, if $\overline{d} = \langle \emptyset \mid a \in \mathcal{P}_{\kappa}\lambda\rangle$, then $B_a = \lambda$ for all $a \in \mathcal{P}_{\kappa}\lambda$, and thus, $F_{d} = I_{\kappa\lambda}^{\ast}$.

Finally, we conclude this section with Theorem~\ref{maintheorem2}. The proof from (1) to (3) requires an interpretation of formulas in \cite{MR3959249}. Here, we note that our proof emphasizes the combinatorial aspects of compactness.

\begin{proof}[Proof of Theorem~\ref{maintheorem2}]
 Note that $\mu^{<\kappa} = \mu$ holds if any of the conditions hold by~\cite{solovay}. 

(1) $\to$ (2) is straightforward. (3) $\leftrightarrow$ (4) follows from Lemma~\ref{jech}. For (3) $\to$ (1), we refer to~\cite{MR3959249}. The new implication is (2) $\to$ (4). For a $\mathcal{P}_{\kappa}\mu$-list $\overline{d}$, we consider the filter $F_{\overline{d}}$, which is generated by at most $\mu^{<\kappa} = \mu$ many sets. By Lemma~\ref{lemma:branchext} and (2), $\overline{d}$ has a branch, as desired. 
\end{proof}

\section{Walks along a weak square sequence}\label{sec:walksquare}

The first half of this section is devoted to constructing a thin $\mathcal{P}_{\kappa}\mu$-list using a $\square(\mu,{<}\kappa)$-sequence. The remainder focuses on the non-semiproperness of $\mathrm{Nm}(\mu)$. The observations obtained in this section will be used to prove Theorem~\ref{maintheorem}. We begin this section with the following theorem\footnote{The author was taught this theorem by Hiroshi Sakai and would like to thank him.}.

\begin{lem}\label{lem:wsqandtp}
 If $\square(\mu,{<}\kappa)$ holds, then $\mathrm{TP}(\kappa,\mu)$ fails for all regular $\kappa \leq \mu$. 
\end{lem}
\begin{proof}
 We follow Sakai's proof~\cite{sakainote}. Let $\overline{C} = \langle C_{\alpha} \mid \alpha < \mu \rangle$ be a $\square(\mu,{<}\kappa)$-sequence. Fix a bijection $l:\mu \times \omega \to \mu$. For all $x \in \mathcal{P}_{\kappa}\mu$, define $d_{x} = l `` (\rho_{2}(-,\sup x) \upharpoonright x) \cap x$. Here, $\rho_2$ is associated with $\overline{C}$. Note that the set $\{x \in \mathcal{P}_{\kappa} \mu\mid l ``x\times \omega = x\}$ is a club. 
\begin{clam}\label{claim:thin}
 $\langle d_{x} \mid x\in \mathcal{P}_{\kappa}\mu \rangle$ is a thin list. 
\end{clam}
\begin{proof}[Proof of Claim]
For a fixed $x \in \mathcal{P}_{\kappa}\mu$, we proceed by induction on $\alpha \in x \cup \{\sup x\}$ to show that $\mathrm{Lev}_{x \cap \alpha}(\overline{d})$ has size ${<}\kappa$ for each $\alpha$. The base case $\mathrm{Lev}_{x \cap \min x}(\overline{d}) = \emptyset$ is trivial. 

For the inductive step, assume that $|\mathrm{Lev}_{x \cap \gamma}(\overline{d})| < \kappa$ for all $\gamma \in x \cap \alpha$. For each $d \in \mathrm{Lev}_{x \cap \alpha}(\overline{d})$, there exists $\beta_{d}$ such that $d \subseteq l ``(\rho_2(-,\beta_{d})) \cap x \cap \alpha$. By Lemma~\ref{walk5}, let $\beta_0 > \cdots > \beta_{m}$ be a walk from $\beta_{d}$ to $\alpha$. Then one of the following holds:
\begin{enumerate}
 \item $\rho_2(\gamma,\beta_{d}) = \rho_2(\gamma,\alpha) + m$ for all $\gamma \in [\overline{\lambda}(\alpha,\beta_d),\alpha)$.
 \item $\rho_2(\gamma,\beta_{d}) = \rho_2(\gamma,\beta_{m-1}) + m - 1$ for all $\gamma \in [\overline{\lambda}(\alpha,\beta_d),\alpha)$. 
\end{enumerate}
If case (1) occurs, then $\alpha \in \mathrm{Lim}(\beta_{m-1})$. Since $(d\cap \overline{\lambda}(\alpha,\beta_{d}))\in \mathrm{Lev}_{x \cap \gamma}(\overline{d})$, we can define an injection from $\mathrm{Lev}_{x\cap \alpha}(\overline{d})$ to the set 
\[
 (\bigcup_{\gamma \in x \cap \alpha}\mathrm{Lev}_{x \cap \gamma}(\overline{d})) \times (\{\beta \mid \alpha \in \mathrm{Lim}(C_{\beta})\} \cup \{\alpha\}) \times \omega.
\] 
By the induction hypothesis and the properties of $\overline{C}$, the size of this set is $< \kappa$, and so $\mathrm{Lev}_{x\cap \alpha}(\overline{d})$ is also $< \kappa$, as desired. 
\end{proof}
Note that the following claim follows directly from the proof of Theorem~\ref{maintheorem}, but we present the proof here for better understanding.

\begin{clam}There is no branch of $\overline{d}$. 
\end{clam}
\begin{proof}[Proof of Claim]
 For contradiction, suppose that there exists a branch $\tilde{d}$. Define $f:\mu \to \omega$ by setting $f(\xi) = n \leftrightarrow l(\xi,n) \in \tilde{d}$. For all $x \in \mathcal{P}_{\kappa}\mu$ with $l ``(x \times \omega) = x$, there exists a $\beta_x$ such that $f\upharpoonright x = \rho_2(-,\beta_x)\upharpoonright x$. For any $\alpha \in E_{{<}\kappa}^\mu$, we fix a cofinal subset $b \in \mathcal{P}_{\kappa}\alpha$ of order type $\mathrm{cf}(\alpha)$. Let $\lambda_{x} = \min b \setminus \overline{\lambda}(\alpha,\beta_{x})$, and for a walk $\beta_0 > \cdots > \beta_{m}$ from $\beta$ to $\alpha$, if $\alpha \in \mathrm{Lim}(C_{\beta_{m-1}})$ then set $m_{x} = \rho_2(\beta_{m-1},\beta_{x}) - 1$, otherwise set $m_x = \rho_2(\alpha,\beta_{x})$. Additionally, let $g(x)$ denote the least $\beta$ such that $C_{\beta_{m-1}} \cap \alpha = C_{\beta} \cap \alpha$ if $\alpha \in \mathrm{Lim}(\beta_{m-1})$. Otherwise, set $g(x)  = \alpha$. 

By $\mathrm{cf}(\mathcal{P}_{\kappa}\alpha,\subseteq) > \kappa > |b| + |\{\beta \mid \alpha \in \mathrm{Lim}(C_\beta)\} \cup \{\alpha\}|$, there exists a cofinal subset $X \subseteq \mathcal{P}_{\kappa}\alpha$, along with $\lambda_\alpha$, $m_{\alpha}$, and $g(\alpha)$, such that 
\[
\forall x \in X(\lambda_\alpha = \lambda_x \land m_\alpha = m_{x} \land g(\alpha) = g(x)).
\]
For all $\gamma < \alpha$ and $x \in X$ with $\gamma \in x$, by Lemma~\ref{walk4}, we have
\[
\gamma \in (C_{g(\alpha)} \cap \alpha) \setminus (\lambda_{\alpha}+1) \leftrightarrow \rho_2(\gamma,\beta_{x}) = m_{\alpha} \land \gamma > \lambda_{\alpha} \leftrightarrow f(\gamma) = m_{\alpha} \land \gamma > \lambda_{\alpha}.
\]
Of course, the function $\alpha \mapsto \lambda_\alpha$ defines a regressive function on $E_{{<}\kappa}^{\mu}$. By Fodor's lemma, we obtain a stationary subset $S \subseteq E_{{<\kappa}}^{\mu}$ and values $\overline{\lambda}, m$ such that
\[
\forall \alpha \in S((C_{g(\alpha)} \cap \alpha) \setminus \overline{\lambda} = \{\gamma < \alpha \mid f(\gamma) = m \land \gamma > \overline{\lambda}\}).
\]
In particular, the set $D = \{\gamma < \mu \mid f(\gamma) = m \land \gamma > \overline{\lambda}\}$ forms a club in $\mu$. For all $\alpha< \beta$ in $T = S \cap \mathrm{Lim}(D)$, it follows that $C_{g(\alpha)} \cap \alpha = C_{g(\beta)}\cap \alpha$. Therefore, the set $C = \bigcup_{\alpha \in T} C_{g(\alpha)} \cap \alpha$ defines a club such that $C \cap \alpha \subseteq C_{g(\min T \setminus \alpha)}$ for all $\alpha$. This contradicts the assumption that $\overline{C}$ is nontrivial. 
\end{proof}
The proof is now complete. 
\end{proof}

\begin{coro}[Hayut~\cite{MR3959249}]
 For regular cardinals $\kappa \leq \lambda \leq \mu \leq 2^\lambda$, if every ${\kappa}$-complete ultrafilter over $\lambda$ generated by $\mu$-many sets can be extended to a $\kappa$-complete ultrafilter, then $\square(\mu,{<}\kappa)$ fails. 
\end{coro}
\begin{proof}
 Note that $\lambda^{<\kappa} = \lambda$ holds by~\cite{solovay} and the given assumption. Lemma~\ref{lem:wsqandtp} and Theorem~\ref{maintheorem2} apply.
\end{proof}

A similar proof shows the following:
\begin{coro}
 For regular cardinals $\kappa \leq \lambda \leq \mu \leq 2^\lambda$, if every ${\kappa}$-complete ultrafilter over $\lambda$ generated by $\mu$-many sets can be extended to a $\delta$-complete ultrafilter, then $\square(\mu,{<}\delta)$ fails. 
\end{coro}
\begin{proof}
 By Lemma~\ref{lem:wsqandtp} and Section~\ref{sec:filter}, we can define a $\kappa$-complete filter $F$ over $\mathcal{P}_{\kappa}\lambda$ using a $\mathcal{P}_{\delta}\mu$-list such that $F$ cannot be extended to a $\delta$-complete ultrafilter. 
\end{proof}

By Lemma~\ref{lem:formodel}, we can view these results as corollaries of Theorem~\ref{maintheorem}.  

In the remainder of this section, we demonstrate how the semiproperness of $\mathrm{Nm}(\mu)$ is destroyed from the perspective of the minimal walk method. Our proof investigates the stationarity of the set $\mathrm{Gal}_{\theta}(\mu,I_{\nu},S_0,\mu)$, and these observations provide insight into Theorem~\ref{maintheorem}. Note that Theorem~\ref{thm:squarefails} follows from~\cite{MR3284479} and Theorem~\ref{thm:locsp}. 
\begin{thm}\label{thm:squarefails}
 For a regular cardinal $\mu \geq \aleph_2$, if $\square(\mu,{<}\aleph_1)$ holds, then $\mathrm{Nm}(\mu)$ is \emph{not} locally semiproper. 
\end{thm}

We fix a sequence of functions $\varphi = \langle \varphi_{\alpha}:\alpha \to \omega_1 \mid \alpha < \mu\rangle$, where each $\varphi_{\alpha}$ is a function from $\mu$ onto $\{0\}$ for each $\alpha < \mu$ with $\mathrm{cf}(\alpha) \not= \omega_1$. Let $\overline{C} = \langle C_\alpha \mid \alpha < \mu \rangle$ be a $\square(\mu,{<}\aleph_1)$-sequence. We introduce a coloring $\rho_{\varphi}:[\mu]^{2} \to \omega_1$. For $\alpha < \beta < \mu$, $\rho_{\varphi}(\alpha,\beta)$ is recursively defined by $\rho_{\varphi}(\alpha,\alpha) = 0$ and 
 \begin{align*}
  {\rho}_{\varphi}(\alpha,\beta) = \sup\{&\varphi_{\beta}(\sup C_{\beta} \cap \alpha), {\rho}_{\varphi}(\alpha,\min (C_{\beta} \setminus \alpha)), \\ 
  & {\rho}_{\varphi}(\xi,\alpha) \mid \xi \in C_{\beta}\cap [\Lambda(\alpha,\beta),\alpha)\}. 
 \end{align*}

Here, $\Lambda(\alpha,\beta) = \max\{\xi \in C_{\beta} \cap (\alpha+1) \mid \omega_1$ divides $\mathrm{ot}(C_{\beta} \cap \xi)\}$. Naturally, $C_{\beta}\cap [\Lambda(\alpha,\beta),\alpha)$ is countable since its order type is a remainder $r$ such that $\mathrm{ot}(C_{\beta}) = q \omega_1 + r$ for some quotient $q < \mu$. Thus, the range of $\rho_{\varphi}$ is $\omega_1$. Our definition of $\rho_\varphi$ is a modification of Todorčević's $\rho_{\omega_1}$~\cite[Section 7]{MR2355670}. If $\varphi_\alpha(\xi) = \mathrm{ot}(C_{\alpha} \cap \xi) \mod \omega_1$ for each $\alpha \in E_{\omega_1}^{\lambda}$, then $\rho_{\varphi}$ coincides with $\rho_{\omega_1}$. The following lemma is central.

\begin{lem}\label{lem:rho2}
For ordinals $\alpha < \delta < \beta < \mu$, if $\rho_{0}(\alpha,\beta) = \rho_0(\delta,\beta) {^{\frown}}\rho_0(\alpha,\delta)$, then $\rho_{\varphi}(\alpha,\beta) \geq \rho_{\varphi}(\alpha,\delta)$. 
\end{lem}
\begin{proof}
The assumption states that $\delta$ appears in the walk from $\beta$ to $\alpha$. Thus, this lemma follows easily.
\end{proof}

From this point, we discuss a specific choice of $\varphi$. For each $\alpha \in E_{\omega_1}^\mu$, let $\langle \varphi_{\xi}^{\alpha} \mid \xi < \omega_1 \rangle$ be an increasing sequence of ordinals converging to $\alpha$. Define $\varphi_{\alpha}(\gamma)$ as the least $\xi< \omega_1$ such that $\gamma \leq \varphi_{\xi}^{\alpha}$. Then the following holds:

\begin{lem}\label{lem:rhomain}
Let $S_0$ be a stationary subset of $[\mu]^{\omega}$ from Lemma~\ref{lem:semistationarysubset}. There exists a club $C \subseteq [\mathcal{H}_{\theta}]^{\omega}$ such that, for every $M \in C$ with $M \cap \mu \in S_0$ and $\beta \geq \sup \mu \cap M_0$, there is an $\alpha \in M \cap \mu$ such that $\rho_{\varphi}(\alpha,\beta) \geq M \cap \omega_1$. 
\end{lem}
\begin{proof}
Consider the expansion $\mathcal{A} = \langle \mathcal{H}_{\theta} ,\in,\overline{C},\varphi, \mu\rangle$. Fix a countable structure $M \prec \mathcal{A}$ with $M \cap \mu \in S$. Letting $\delta = \sup M \cap \mu$, for all $\beta \geq \delta$ with $\delta \in \mathrm{Lim}(C_{\beta} \cup \{\sup C_{\beta}\})$, we have that $M \cap C_{\beta}$ is bounded in $\delta$ and, for all sufficiently large $\gamma \in C_{\beta}$, $\mathrm{cf}(\min (M \cap \mu) \setminus \gamma) = \omega_1$. 

For every $\tilde{\beta} > \delta$, we claim that there exists an $\alpha \in M \cap \mu$ such that $\rho_{\varphi}(\alpha,\tilde{\beta}) \geq M \cap \omega_1$. Let $\beta_{0} > \cdots > \beta_{m}$ be a walk from $\tilde{\beta}$ to $\delta$. Define $\beta$ by $\beta = \beta_{m-1}$ if $\delta \in \mathrm{Lim}(C_{\beta_{m-1}})$. Otherwise, let $\beta = \delta$. 

Since $C_{\beta} \cap \delta$ is unbounded in $\delta$, by the choice of $M$, we can select $\alpha,\gamma$ such that:
\begin{itemize}
 \item $\overline{\lambda}(\delta,\tilde{\beta}) < \delta < \sup (M \cap \alpha) \leq \gamma < \alpha = \min (M \cap \mu) \setminus \gamma$,
 \item $\mathrm{cf}(\alpha) = \omega_1$, and 
 \item $\gamma \in C_{\beta}$. 
\end{itemize}
Indeed, we can choose $\gamma \in C_{\beta}$ such that $\gamma \geq \sup (M \cap \gamma) > \mathrm{\lambda}(\delta,\tilde{\beta})$, and let $\alpha = \min (M \cap \mu) \setminus \gamma$ be of cofinality $\omega_1$. Note that there is no ordinal between $\alpha$ and $\sup (M \cap \alpha)$, so $\sup (M \cap \alpha) = \sup (M \cap \gamma)$, as required. 

By $\overline{\lambda}(\delta,\tilde{\beta}) < \alpha$ and the definition of $\beta$, we have $\rho_{0}(\alpha,\tilde{\beta}) = \rho_0(\beta,\tilde{\beta}){^{\frown}} \rho_0(\alpha,\beta)$. By Lemma~\ref{lem:rho2}, $\rho_{\varphi}(\alpha,\tilde{\beta}) \geq \rho_{\varphi}(\alpha,\beta)$. We now evaluate $\rho_{\varphi}(\alpha,\beta)$. Note that $\omega_1$ does not divide $\mathrm{ot}(C_{\beta} \cap \alpha)$. Indeed, since $\alpha \not\in C_{\beta}$, $C_{\beta} \cap \alpha$ has a maximal element $\xi = \sup (C_{\beta} \cap \alpha) \in C_{\beta} \cap \alpha$. Therefore, $\mathrm{ot}(C_{\beta} \cap \alpha) = \omega_1\cdot q + r' + 1$ for some countable ordinal $r'$. Then, $\xi \in C_{\alpha} \cap [\Lambda(\alpha,\beta),\alpha)$. By the definition of $\rho_{\varphi}(\alpha,\beta)$, it follows that $\rho_{\varphi}(\alpha,\beta) \geq \rho_{\varphi}(\xi,\alpha)$. 

Now, we note that $\alpha > \xi \geq \gamma \geq \sup (M \cap \alpha) = \sup (M \cap \gamma) = \sup (M \cap \xi)$. Since $\alpha \in M$, it follows that $\varphi_{\alpha}(\gamma') \in M \cap \omega_1$ for all $\gamma' \in M \cap \alpha$, and thus:
\[
\sup_{\gamma'\in M \cap \alpha} \varphi_{\alpha}(\gamma') = M \cap \omega_1.
\]
Hence, we have:
\[
\rho_{\varphi}(\xi,\alpha) \geq \varphi_{\alpha}(\sup C_{\alpha} \cap \xi) \geq \varphi_{\alpha}(\sup (M \cap \alpha))  \geq \sup_{\gamma' \in M \cap \alpha}\varphi_{\alpha}(\gamma') = M \cap \omega_1.
\]
Thus, in particular, we obtain:
\[
\rho_{\varphi}(\alpha,\tilde\beta) \geq \rho_{\varphi}(\alpha,\beta) \geq \rho_{\varphi}(\xi,\alpha) \geq M \cap \omega_1,
\]
as required.
\end{proof}

\begin{proof}[Proof of Theorem~\ref{thm:squarefails}]
Let $S_0$ be the stationary subset given by Lemma~\ref{lem:semistationarysubset}. Let $C$ be the club set from Lemma~\ref{lem:rhomain}. We claim that $\mathrm{Gal}_{\theta}(\mu,I_\mu,S_0,\mu) \cap C = S_{0}^{\uparrow \mathcal{H}_{\theta}} \cap C$. 

Let $M \in C$ be an elementary substructure such that $M \cap \mu \in S_0$. For each $\beta \geq \sup (M \cap \mu)$, by Lemma~\ref{lem:rhomain}, there exists an $\alpha \in M \cap \mu$ such that $\rho_\varphi(\alpha,\beta) \geq \sup (M \cap \omega_1)$. Note that $\rho_{\varphi}$ is definable from $\overline{C}$ and $\varphi$. If we set $f = \rho_{\varphi}(\alpha,-)$, then $f \in M$ for all $\alpha \in M \cap \mu$. Therefore, 
\[
\bigcap\{f^{-1}M \cap \omega_1 \mid f:\mu \to \omega_1 \in M\} \subseteq \delta.
\]
Thus, $M \in \mathrm{Gal}_{\theta}(\mu,I_\mu,S_0,\mu)$. By Lemma~\ref{lem:gal3}, $S_0$ is not preserved by $\mathrm{Nm}(\mu)$, as desired.
\end{proof}

\begin{rema}
The proof of Theorem~\ref{thm:squarefails} shows that, by the proof of Lemma~\ref{lem:gal3}, $\mathrm{Nm}(\mu) \force S_0$ is \emph{not} semistationary.
\end{rema}

Note that this proof includes Lemma~\ref{lem:defctblord}, which will be used in Theorem~\ref{maintheorem}. 

\begin{lem}\label{lem:defctblord}
There exists a club $C \subseteq [\mathcal{H}_{\theta}]^{\omega}$ such that, if $M \in C$ and $\alpha \in M \cap E_{\omega_1}^{\mu}$, then $\mathrm{Sk}(M \cup \{\zeta\})\cap \omega_1 > M \cap \omega_1$ for all $\zeta \in [\sup (M \cap \alpha),\alpha)$. 
\end{lem}

\begin{rema}
We briefly discuss the necessity of the modification. If we assume that $\overline{C}$ is a special $\square(\mu)$-sequence, meaning that $\mathrm{ot}(C_{\alpha}) = \omega_1$ for stationary many $\alpha$, then we can use $\rho_{\omega_1}$ instead of $\rho_\varphi$ to prove Theorem~\ref{thm:squarefails}. Todorčević showed that $\rho_{\omega_1}$ satisfies useful properties such as subadditivity without assuming speciality. However, in order to construct an Aronszajn-like tree, he still required some form of speciality in the $C$-sequence. It remains unclear whether $\rho_{\omega_1}$ works when $\overline{C}$ is not special. For more details, we refer to~\cite[Section 7.2]{MR2355670}.
\end{rema}

\section{Theorems~\ref{maintheorem} and~\ref{maintheorem:model}}\label{sec:maintheorem}
This section is devoted to Theorems~\ref{maintheorem} and~\ref{maintheorem:model}. First, we prove Theorem~\ref{maintheorem}. After that, we discuss its applications, including Theorem~\ref{maintheorem:model}. 

\begin{proof}[Proof of Theorem~\ref{maintheorem}]
Let $\overline{d}$ be the thin $\mathcal{P}_{\kappa}\mu$-list defined in Section~\ref{sec:walksquare}. The components used to define $\overline{d}$ include a $\square(\mu,{<}\aleph_1)$-sequence $\overline{C} = \langle C_{\alpha} \mid \alpha < \mu \rangle$\footnote{Note that a $\square(\mu,{<}\aleph_1)$-sequence is also a $\square(\mu,{<}\kappa)$-sequence since $\aleph_2 \leq \kappa$.} and a bijection $l:\mu\times \omega \to \mu$. 

From Section~\ref{sec:filter}, we obtain a $\kappa$-complete fine filter $F_{\overline{d}}$ over $\lambda$ using a family $\overline{A} = \langle A_{\alpha} \mid \alpha < \mu \rangle$ from Lemma~\ref{lemma:propvaluesets}. Let $S_0$ and $S_1$ be the stationary subsets given by Lemmas~\ref{lem:semistationarysubset} and~\ref{lem:wqkillstationary}. Additionally, let $E$ be the club set given by Lemma~\ref{lem:defctblord}. 

We recall the definitions of $S_0$ and $S_1$:
\begin{itemize}
    \item $S_0 \subseteq [\mu]^{\omega}$ consists of all countable $x \subseteq \mu$ such that for every $\beta \geq \delta = \sup (x \cap \mu)$ with $\delta \in \mathrm{Lim}(C_{\beta})$, 
    \begin{itemize}
        \item $x \cap C_{\beta}$ is bounded in $\delta$.
        \item For all sufficiently large $\gamma \in C_{\beta} \cap \delta$, $\mathrm{cf}(\min ((x \cap \mu) \setminus \gamma)) = \omega_1$. 
    \end{itemize}
    \item $S_1 \subseteq [\mathcal{H}_{\theta}]^{\omega}$ consists of all countable $M \prec \mathcal{H}_{\theta}$ such that for every $\beta \geq \delta = \sup (M \cap \mu)$ with $\delta \in \mathrm{Lim}(C_{\beta})$,
    \begin{itemize}
        \item $M \cap C_{\beta}$ is bounded in $\delta$.
        \item For all sufficiently large $\gamma \in C_{\beta} \cap \delta$, $\mathrm{cf}(\min ((M \cap \mu) \setminus \gamma)) = \omega_1$. 
        \item For each $\xi < \delta$, there exists an $x \in \mathcal{P}_{\kappa}\mu \cap M$ such that $(x \cap C_{\beta} \cap \delta) \setminus \xi \neq \emptyset$. 
    \end{itemize}
\end{itemize}
By Lemma~\ref{lem:wqkillstationary}, both $S_0$ and $S_1$ are stationary. 

By Lemma~\ref{lem:gal3}, it suffices to show that $\mathrm{Gal}_{\theta}(\mathcal{P}_{\kappa}\lambda,F_{\overline{d}}, S_0, \mathcal{P}_{\kappa}\lambda)$ is stationary. Consider the expansion $\mathcal{A} = \langle \mathcal{H}_{\theta}, \in, \overline{C}, E, \kappa, \lambda, \mu, F_{\overline{d}}, \overline{A}, l \rangle$. Let $D$ be the set of all countable elementary substructures of $\mathcal{A}$. We aim to show that 
\[
 D \cap S_{1} \subseteq \mathrm{Gal}_{\theta}(\mathcal{P}_{\kappa}\lambda,F_{\overline{d}}, S_0, \mathcal{P}_{\kappa}\lambda). 
\]

Fix $M \in D \cap S_1$. Then $M \cap \mu \in S_0$. Let $\delta = \sup (M \cap \mu)$. There exists an $x \in \mathcal{P}_{\kappa}\mu$ such that $l ``(x\times \omega) = x$ and $\bigcup(\mathcal{P}_{\kappa}\mu \cap M) \subseteq x$. Since $B_{x} \in F_{\overline{d}}$, we need to show that $B_{x} \cap \bigcap \{f^{-1}M \cap \omega_1 \mid f:\mathcal{P}_{\kappa}\lambda \to \omega_1 \in M\} = \emptyset$. 

For all $\tilde{x} \in B_{x}$, since $\bigcap \{f^{-1}M \cap \omega_1 \mid f:\mathcal{P}_{\kappa}\lambda \to \omega_1 \in M\} = \{\tilde{x} \in \mathcal{P}_{\kappa}\lambda \mid \mathrm{Sk}(M \cup \{\tilde{x}\}) \cap \omega_1 = M \cap \omega_1\}$, it suffices to show that $\mathrm{Sk}(M \cup \{\tilde{x}\}) \cap \omega_1 > M \cap \omega_1$. Let $N = \mathrm{Sk}(M \cup \{\tilde{x}\})$.

Since $\tilde{x} \in B_{x}$, there exists a $d \in \mathrm{Lev}_{x}(\overline{d})$ such that $\tilde{x} \in \bigcap_{\alpha \in d \cap x}A_{\alpha} \setminus \bigcup_{\alpha \in x \setminus d}A_{\alpha}$. We note that $l ``(x\times \omega) = x$. Define a partial function $f:\mu \to \omega$ by 
\begin{center}
 $f(\xi) = n \leftrightarrow \tilde{x} \in A_{l(\xi,n)}$ and there is no $m$ such that $n\neq m \land \tilde{x} \in A_{l(\xi,m)}$.
\end{center} 

Since $f$ is defined from $\tilde{x}$, $l$, and $\overline{A}$, we conclude that $f \in N$. By the choice of $f$ and $x$, we have $x \subseteq \mathrm{dom}(f)$ and $f \upharpoonright x = \rho_{2}(-,\tilde{\beta}) \upharpoonright x$ for some $\tilde{\beta} \geq \sup x$. Let $\beta_{0} > \cdots > \beta_{m}$ be a walk from $\tilde{\beta}$ to $\delta$. If $\delta \in \mathrm{Lim}(\beta_{m-1})$, we set $\beta = \beta_{m-1}$. Otherwise, we set $\beta = \delta$. By Lemma~\ref{walk4}, letting $n = \rho_2(\delta,\tilde{\beta})$ or $\rho_2(\delta,\tilde{\beta})-1$ and $\overline{\lambda} = \overline{\lambda}(\delta,\tilde{\beta}) < \delta$, we have:
\[
  (C_{\beta} \cap \delta)\setminus (\overline{\lambda} + 1) = \{\gamma < \delta \mid \rho_2(\gamma,\tilde{\beta}) = n\}\setminus (\overline{\lambda} + 1).
\]
Since $M \cap \mu \in S_0$, we know that $C_{\beta} \cap M \cap \mu$ is bounded in $\delta$. On the other hand, we can show that cofinally many elements of $C_{\beta} \cap \delta$ are newly appearing, as follows. 

\begin{clam}
 $C_{\beta} \cap N \cap \mu$ is cofinal in $\delta$. 
\end{clam}
\begin{proof}[Proof of claim]
Observe that $\sup (N \cap \mu) = \sup (M \cap \mu)$. Thus, it suffices to show that for all $\zeta \in M \cap \mu$, we can find $\gamma \in C_{\delta} \cap N \cap \mu$ greater than $\zeta$. We may assume that $\zeta > \sup (C_{\delta} \cap M)$. By the definition of $S_1$, there exist $y$ and $\gamma \in C_{\delta}$ satisfying:
\begin{itemize}
 \item $\gamma \in y \in M \cap \mathcal{P}_{\kappa}\mu$. 
 \item $\gamma > \max\{\overline{\lambda}, \zeta\}$. 
\end{itemize}
 Of course, $\rho_2(\gamma,\tilde{\beta}) = n$. In $N$, the set $(C_{\tilde{\beta}} \cap y) \setminus \zeta$ can be defined as 
\[
\{\gamma \in (y\setminus \zeta) \mid f(\gamma) = n\}.
\]
 Note that for all $\xi \in M \cap (y \setminus \zeta)$, we have $\rho_2(\xi,\tilde{\beta}) = f(\xi) > n$, since $\xi \not \in C_{\beta}$. Clearly, $(C_{\beta} \cap y) \setminus \zeta$ is nonempty due to the existence of $\gamma$. By elementarity, we can choose an element from $(C_{\beta} \cap y) \setminus \zeta$ within $N$, as required. 
\end{proof}

By the claim and Lemma~\ref{lem:semistationarysubset}, we can choose $\gamma \in N \cap C_{\beta}$ such that the cofinality of $\alpha = \min (M \setminus \gamma)$ is $\omega_1$ and $\gamma \not\in M$. Clearly, $\gamma \in [\sup M\cap \alpha,\alpha)$. By Lemma~\ref{lem:defctblord}, we obtain:

\[
N \cap \omega_1 \geq \mathrm{Sk}(M \cup \{\gamma\}) \cap \omega_1 > M \cap \omega_1. 
\]
Thus, $M \in \mathrm{Gal}_{\theta}(\mathcal{P}_{\kappa}\lambda,F_{\overline{d}},S_{0},\mathcal{P}_{\kappa}\lambda)$. Since $M$ was chosen arbitrarily from $D \cap S_1$, we conclude that $\mathrm{Gal}_{\theta}(\mathcal{P}_{\kappa}\lambda,F_{\overline{d}},S_{0},\mathcal{P}_{\kappa}\lambda)$ contains a stationary subset $D \cap S_1$. By Lemma~\ref{lem:gal3}, $S_0$ contains a stationary subset $T$ such that $\mathrm{Nm}(\kappa,\lambda,F_{\overline{d}})$ forces that $T$ is \emph{non}-semistationary, as desired. 
\end{proof}

\begin{coro}
 For regular cardinals $\aleph_2 \leq \kappa \leq \lambda = \lambda^{<\kappa} < \mu \leq 2^{\lambda}$, if $\mathrm{Nm}(\kappa,\lambda,F)$ is semiproper for every $\kappa$-complete fine filter $F$ over $\mathcal{P}_{\kappa}\lambda$ generated by $\mu^{<\kappa}$-many sets, then $\square(\nu)$ fails for all regular $\nu \in [\lambda,\mu]$. 
\end{coro}

As an application of Theorem~\ref{maintheorem}, we analyze the consistency strength.

\begin{coro}
Suppose that $\aleph_2 \leq \kappa \leq \lambda = \lambda^{<\kappa} \leq \mu \leq 2^{\lambda}$ are regular cardinals and $\mathrm{Nm}(\kappa,\lambda,F)$ is semiproper for every $\kappa$-complete fine filter $F$ over $\mathcal{P}_{\kappa}\lambda$. Then the following holds:
\begin{enumerate}
 \item If $\aleph_3 \leq \kappa$ and $\forall \nu < \kappa(\nu^{\aleph_0} < \kappa)$, then there is an inner model with a proper class of strong cardinals and a proper class of Woodin cardinals. 
 \item If $2^{\aleph_0} \leq \kappa$, then there is an inner model with infinitely many Woodin cardinals. 
\end{enumerate}
\end{coro}
\begin{proof}
 By Theorem~\ref{maintheorem}, both $\square(\kappa)$ and $\square(\kappa^{+})$ fail. Theorems~\ref{sqinnermodel2} and~\ref{sqinnermodel1} establish (1) and (2), respectively. 
\end{proof}

\begin{coro}\label{coro:consistencystrength}
 Suppose that $2^{\aleph_1} = \aleph_2$ and $\mathrm{Nm}(\aleph_2,F)$ is semiproper for all $\aleph_2$-complete fine filters $F$ over $\aleph_2$ generated by $\aleph_3$ sets. Then both $\square(\aleph_2)$ and $\square(\aleph_3)$ fail. In particular, there is an inner model with infinitely many Woodin cardinals. 
\end{coro}
\begin{proof}
 By Theorem~\ref{thm:todch}, we have $2^{\aleph_0} \leq \aleph_2$. Since $2^{\aleph_1} = \aleph_2$, we note that $\aleph_{i}^{<\aleph_2} = \aleph_{i}$ for $i \in \{2,3\}$. Theorems~\ref{maintheorem} and~\ref{sqinnermodel2} complete the proof. 
\end{proof}

We conclude this section with a proof of Theorem~\ref{maintheorem:model}.

\begin{proof}[Proof of Theorem~\ref{maintheorem:model}]
 By the assumption, we note that $\lambda^{<\kappa} = \lambda$ holds by~\cite{solovay}. 
 Let $U$ be a $\kappa$-complete fine ultrafilter over $\mathcal{P}_{\kappa}\lambda$. Then the standard Lévy collapse $\mathrm{Coll}(\mu,{<}\kappa)$ forces, letting $\dot{F}$ be a $\mathrm{Coll}(\mu,{<}\kappa)$-name for the filter generated by $U$, that $\mathcal{B}(\dot{F}^{+},\subseteq) \simeq \mathcal{B}(\mathrm{Coll}(\mu,{<}\kappa'))$ for some $\kappa' > \kappa$ by \cite[Theorem 7.14]{MR2768692}. Since $\mathrm{Coll}(\mu,{<}\kappa)$ is $\mu$-closed and $\mu \geq \aleph_1$, $\langle \dot{F}^{+},\subseteq \rangle$ is forced to have an $\aleph_1$-closed dense subset. By (1) of Lemma~\ref{lem:formodel}, $\mathrm{Coll}(\mu,{<}\kappa)$ forces that ${\mathrm{Nm}}(\kappa,\lambda,\dot{F})$ is semiproper. Moreover, $\mathrm{Nm}(\kappa,\lambda)$ is forced to be semiproper. 

Let $\dot{Q}$ be a $\mathrm{Coll}(\mu,{<}\kappa)$-name for a $\lambda+1$-strategically closed poset that adds a $\square(\lambda^{+})$-sequence~\cite[Definition 6.2]{MR1838355}. By (2) of Lemma~\ref{lem:formodel}, $\mathrm{Coll}(\mu,{<}\kappa)\ast \dot{Q} \force \mathrm{Nm}(\kappa,\lambda,\dot{F})$ is semiproper. On the other hand, since the existence of a $\square(\lambda^{+})$-sequence is forced, Theorem~\ref{maintheorem} shows the existence of a fine $\kappa$-complete filter $\dot{F}'$ over $\mathcal{P}_{\kappa}\lambda$ such that $\mathrm{Nm}(\kappa,\lambda,\dot{F}')$ is \emph{not} semiproper. 
 
 In particular, $P = \mathrm{Coll}(\mu,{<}\kappa)\ast \dot{Q}$ is a required poset.
\end{proof}

Note that any stationary subset of $S_0 \subseteq [\mu]^{\omega}$ does not reflect to any $a \in \mathcal{P}_{\kappa}\mu$ by Theorems~\ref{thm:locsp} and~\ref{thm:squarefails}. However, semistationarity of any subset of $S_0$ may be preserved by Namba forcings, while preservability by $\mathrm{Nm}(\kappa,\lambda,F)$ and reflectability of subsets of $[\lambda]^{\omega}$ are equivalent, as seen in Theorem~\ref{thm:locsp}. 

\begin{prop}
In the extension by $P$ given in Theorem~\ref{maintheorem:model}, let $S_0$ be a stationary set from Lemma~\ref{lem:semistationarysubset}. The following holds:
 \begin{enumerate}
  \item For every $\kappa$-complete fine filter $F'$ over $\mathcal{P}_\kappa\lambda$, if $F'$ is generated by $<\lambda$-many sets, then $\mathrm{Nm}(\lambda,F') \force T$ is semistationary for all stationary subsets $T \subseteq S_0$. In particular, $\mathrm{Nm}(\kappa,\lambda) \force T$ is semistationary. 
  \item There exist a $\kappa$-complete fine filter $F$ over $\mathcal{P}_{\kappa}\lambda$ and a stationary subset $T \subseteq S_0$ such that $\mathrm{Nm}(\lambda,F) \force T$ is \emph{non}-semistationary. 
 \end{enumerate}
\end{prop}

Lastly, we point out that $\aleph_1$ in Theorem~\ref{maintheorem} cannot be improved to $\aleph_3$. 

\begin{prop}
 Suppose that $\kappa$ is supercompact and $\nu \geq \kappa$ is regular. For a singular $\mu$ with $\mathrm{cf}(\mu) = \nu$, there is a generic extension such that:
 \begin{enumerate}
  \item For every $\nu$-complete fine filter $F$ over $\mathcal{P}_{\nu}\lambda$, $\mathrm{Nm}(\nu,\lambda,F)$ is semiproper for all $\lambda \geq \nu^{+}$. Indeed, $(\dagger)$ holds. 
  \item $\square(\mu,{<}\nu)$ holds.
 \end{enumerate}
\end{prop}

\begin{proof}
 First, by~\cite{indlaver}, we may assume that $\kappa$ is indestructible. That is, for every $\kappa$-directed closed poset $P$, we have $P\force \kappa$ remains supercompact. By \cite[Theorem 16]{MR1838355}, we have a poset $P$ with the following properties:
\begin{itemize}
 \item $P$ is $\nu$-directed closed.
 \item $P$ is $\alpha$-strategically closed for every $\alpha < \mu$.
 \item $P \force \square(\mu^{+},{<}\nu^{+})$. 
\end{itemize}
By the first item and the preparation, $P$ forces that $\kappa$ remains supercompact. Let us now analyze the extension by $P$. Let $\overline{C}$ be a $\square(\mu^{+},{<}\nu^{+})$-sequence. Since $\mathrm{Coll}(\aleph_1,{<}\kappa)$ forces $(\dagger)$, it suffices to show that $\mathrm{Coll}(\aleph_1,{<}\kappa)$ forces that $\overline{C}$ is a $\square(\mu^{+},{<}\nu^{+})$-sequence. 

By $\nu \geq \kappa$, condition (1) of the definition of $\square(\mu^{+},{<}\nu^{+})$ is satisfied. For condition (2), fix $p \in \mathrm{Coll}(\aleph_1,{<}\kappa)$ and a name $\dot{C}$ for a club in $\mu^{+}$. Since $\mathrm{Coll}(\aleph_1,{<}\kappa)$ has the $\mu^{+}$-c.c., a standard argument ensures that we can find a club $D \subseteq \mu^{+}$ such that $p \force D \subseteq \dot{C}$. Since $\overline{C}$ is non-trivial, there exists an $\alpha < \mu^{+}$ such that $D \cap \alpha \not\subseteq C_\beta$ for all $\beta \geq \alpha$, implying that $p \force \dot{C} \cap \alpha \not\subseteq C_{\beta}$ for all $\beta \geq \alpha$. 

Thus, a generic extension by $\mathrm{Coll}(\aleph_1,{<}\kappa)$ provides the desired model. 
\end{proof}

In particular, $\square(\mu,{<}\aleph_3)$ does not necessarily follow from the semiproperness of Namba forcings over $\mathcal{P}_{\aleph_3}\mu$. It remains an open question\footnote{This question is observed in~\cite{torresperezandwu} and~\cite{MR3959249}} whether $(\dagger)$ implies $\square(\mu,{<}\aleph_2)$. 

\section{Stationary reflections and $\mathrm{SCH}$}\label{sec:reflections}
This section is devoted to two theorems and related observations. The first theorem is an analogue of Theorem~\ref{maintheorem} in the context of stationary reflection principles. The second part of the section is a study of the Singular Cardinal Hypothesis ($\mathrm{SCH}$).

\begin{thm}\label{thm:additional1}
 For regular cardinals $\aleph_2 \leq \kappa \leq \lambda = \lambda^{<\kappa} < \mu \leq 2^{\lambda}$, if $\mathrm{Nm}(\lambda,F)$ is semiproper for every $\kappa$-complete fine filter $F$ over $\mathcal{P}_{\kappa}\lambda$ generated by $\mu^{<\kappa}$-many sets and $\kappa$ is inaccessible, then $\mathrm{Refl}(E_{\omega}^{\mu})$ holds.
\end{thm}

\begin{proof}

We show the contraposition. Suppose that there exists a stationary subset $S \subseteq E_{\omega}^{\mu}$ such that $S \cap \delta$ is \emph{non}-stationary for all $\delta \in \mathrm{Lim} \cap \mu$. For every $\delta \in \mathrm{Lim} \cap \mu$, let $C_{\delta}$ be a club such that $C_{\delta} \cap S = \emptyset$. For a successor ordinal $\alpha+1$, define $C_{\alpha+1}= \{\alpha\}$. We consider a $C$-sequence $\overline{C} = \langle C_{\alpha} \mid \alpha < \mu \rangle$. 

\begin{clam}\label{addthm:claim1}
 $\overline{C}$ is non-trivial. 
\end{clam}
\begin{proof}[Proof of Claim]
 For every club $C$, we fix an $\alpha \in \mathrm{Lim}(C) \cap S$ such that $\sup C \cap \alpha \cap S =\alpha$. For every $\beta \geq \alpha$, since $|C_{\beta} \cap S| \leq 1$, we can deduce $C \cap \alpha \not\subseteq C_{\beta}$, as desired. 
\end{proof}
 
\begin{clam}\label{addthm:claim2}
 For every $\alpha < \beta < \mu$, if $\alpha \in S$, then $\lambda(\alpha,\beta) = \overline{\lambda}(\alpha,\beta)$. 
\end{clam}
\begin{proof}[Proof of Claim]
 Let $\beta_0 > \cdots > \beta_{n}$ be a walk from $\beta$ to $\alpha$. Since $\alpha \in S$, we have $\alpha\in C_{\beta_{i}}$ for all $i < n$. In particular, $\sup C_{\beta_{i}} \cap \alpha < \alpha$ for all $i< n$. Therefore, $\lambda(\alpha,\beta) < \alpha$, and thus, $\lambda(\alpha,\beta) = \overline{\lambda}(\alpha,\beta)$. 
\end{proof}

Fix a bijection $l:\mu \times \omega \to \mu$. Define a $\mathcal{P}_{\kappa}\mu$-list $\overline{d} = \langle d_{x} \mid x \in \mathcal{P}_{\kappa}\mu \rangle$ by 
\[
d_{x} = (l `` \rho_2(-,\min (S \setminus \sup x))\upharpoonright x)\cap x.
\]
Since $\kappa$ is inaccessible\footnote{This is the only point where $\kappa$ needs to be inaccessible.}, $\overline{d}$ is thin. 

Let $\overline{A} = \langle A_{\alpha}\mid \alpha < \mu\rangle$ be a family of subsets of $\mathcal{P}_{\kappa}\lambda$ obtained from Lemma~\ref{lemma:propvaluesets}. Consider the filter $F_{\overline{d}}$ over $\lambda$. Since $\overline{d}$ is thin, $F_{\overline{d}}$ is $\kappa$-complete. 

Let $S_0$ be the set of all $x \in [\mu]^\omega$ with the following properties:
 \begin{itemize}
  \item $x \cap C_{\sup x \cap \mu}$ is bounded in $\sup x\cap \mu$. 
  \item For all sufficiently large $\gamma \in C_{x \cap \mu}$, $\mathrm{cf}(\min x \setminus \mu) = \omega_1$. 
 \end{itemize}

Let $S_1 \subseteq [\mathcal{H}_{\theta}]^{\omega}$ be the set of all countable $M \prec \mathcal{H}_{\theta}$ such that 
       \begin{itemize}
	\item $M \cap C_{\sup M \cap \mu}$ is bounded in $\sup M\cap \mu$. 
	\item For all sufficiently large $\gamma \in C_{\sup M \cap \mu}$, $\mathrm{cf}(\min ((M \cap \mu) \setminus \gamma)) = \omega_1$. 
	\item For each $\xi < \sup M \cap\mu$, there exists an $x \in \mathcal{P}_{\kappa}\mu \cap M$ such that $(x \cap C_{\sup M \cap \mu}) \setminus \xi \neq \emptyset$. 
       \end{itemize}

By Lemma~\ref{lem:killstationary}, both $S_0$ and $S_1$ are stationary.

Let $\mathcal{A} = \langle \mathcal{H}_{\theta} ,\in,\overline{C}, E, S,\kappa,\lambda, \mu, F_{\overline{d}}, \overline{A},l\rangle$. Let $D$ be the set of all elementary substructures of $\mathcal{A}$. We aim to show that 
\[
 S_1 \cap D \subseteq \mathrm{Gal}_{\theta}(\mathcal{P}_{\kappa}\lambda,F_{\overline{d}},S_0,\mathcal{P}_{\kappa}\lambda).
\]

Take an $M \in D \cap S_1$. Then $M \cap \mu \in S_0$. Let $\delta = \sup M \cap \mu$. Note that there exists an $x \in \mathcal{P}_{\kappa}\mu$ such that $l ``(x\times \omega) = x$ and $\bigcup(\mathcal{P}_{\kappa}\mu \cap M)\subseteq x$. Since $B_{x} \in F_{\overline{d}}$, we must verify that 
\[
B_{x} \cap \bigcap \{f^{-1}M \cap \omega_1 \mid f:\mathcal{P}_{\kappa}\lambda \to \omega_1 \in M\} = \emptyset.
\]

For all $\tilde{x} \in B_{x}$, let $N = \mathrm{Sk}(M \cup \{\tilde{x}\})$. It suffices to show that $N \cap \omega_1 > M \cap \omega_1$. The proof of Theorem~\ref{maintheorem} establishes that $C_{\delta} \cap N \cap \mu$ is cofinal in $\delta$ once again. 

By the definition of $S_1$, we can choose $\gamma \in N \cap C_{\delta}$ such that the cofinality of $\alpha = \min (M \setminus \gamma)$ is $\omega_1$ and $\gamma \not\in M$. Clearly, $\gamma \in [\sup M\cap \alpha,\alpha)$. By Lemma~\ref{lem:defctblord}, we obtain
\[
N \cap \omega_1 \geq \mathrm{Sk}(M \cup \{\gamma\}) \cap \omega_1 > M \cap \omega_1.
\]
Thus, $M \in \mathrm{Gal}_{\theta}(\mathcal{P}_{\kappa}\lambda,F_{\overline{d}}, S_0,\mathcal{P}_{\kappa}\lambda)$, as required.

Therefore, $\mathrm{Gal}_{\theta}(\mathcal{P}_{\kappa}\lambda,F_{\overline{d}},S_0,\mathcal{P}_{\kappa}\lambda)$ is stationary. By Lemma~\ref{lem:gal3}, there exists a stationary subset $T \subseteq S_0$ such that $\mathrm{Nm}(\kappa,\lambda,F_{\overline{d}})$ forces that $T$ is \emph{non}-semistationary, completing the proof.\end{proof}

It seems that the assumption that $\kappa$ is inaccessible is not needed. Indeed, if $\mu = \lambda$, then we do not need this assumption (see Theorem~\ref{thm:twowalks2}). However, for $\mu > \lambda$, our proof relies on the existence of a $\mathcal{P}_{\kappa}\mu$-thin tree to define a filter. It is known that the existence of a \emph{non}-reflecting stationary subset is compatible with the tree property. Therefore, $\overline{d}$ may \emph{not} be a thin list.

\begin{prop}\label{prop:pfaandstt}
 If a supercompact cardinal exists, then there is a generic extension in which $\mathrm{TP}(\aleph_2,\lambda)$ holds for all $\lambda \geq \aleph_2$, and there exists a \emph{non}-reflecting stationary subset $S \subseteq E_{\omega}^{\aleph_2}$. In this model, for every $C$-sequence that avoids $S$, a $\mathcal{P}_{\aleph_2}\mu$-list $\overline{d}$, defined as in Theorem~\ref{thm:additional1}, is \emph{not} thin.
\end{prop}
\begin{proof}
 By~\cite[Theorem 4.8]{consistencypfa}, we know that $\mathrm{PFA}$ implies $\mathrm{TP}(\aleph_2,\lambda)$ holds for all $\lambda \geq \aleph_2$. In~\cite{beaudoin1991proper}, a generic extension was constructed in which $\mathrm{PFA}$ holds but a \emph{non}-reflecting stationary subset $S \subseteq E_{\omega}^{\aleph_2}$ exists. This model satisfies the required conditions.

For a $C$-sequence $\overline{C} = \langle C_\alpha \mid \alpha < \mu \rangle$ that avoids $S$, let $\overline{d}$ be a $\mathcal{P}_{\aleph_2}\mu$-list defined as in Theorem~\ref{thm:additional1}. If $\overline{d}$ is thin, then it must have a branch. However, as in Lemma~\ref{lem:wsqandtp}, we can define a club $C$ such that for all $\alpha < \mu$, there is some $\beta \geq \alpha$ such that $C \cap \alpha \subseteq C_\beta$. This contradicts the assumption, completing the proof.
\end{proof}

We should remark that there are no implications between the semiproperness of $\mathrm{Nm}(\lambda,F_{\overline{d}})$ and the existence of a branch of $\overline{d}$. 

It is straightforward to construct a model where $\mathrm{Nm}(\lambda, F_{\overline{d}})$ is \emph{not} semiproper, but $\overline{d}$ has a branch. Indeed, a thin $\mathcal{P}_{\kappa}\lambda$-list $\overline{d}$ with a branch always exists. On the other hand, obtaining the semiproperness of $\mathrm{Nm}(\lambda, F_{\overline{d}})$ requires the failure of $\square(\lambda)$. We also verify the following proposition.

\begin{prop}
It is consistent that there exists a thin $\mathcal{P}_{\kappa}\lambda$-list $\overline{d}$ such that $\overline{d}$ has \emph{no} branch but $\mathrm{Nm}(\lambda, F_{\overline{d}})$ is semiproper.
\end{prop}
\begin{proof}
Every $\mathrm{Nm}(\lambda, F_{\overline{d}})$ is $\omega_1$-stationary preserving. It is known that if a strongly compact cardinal $\kappa$ is collapsed to $\aleph_2$ by the standard L\'evy collapse $\mathrm{Coll}(\aleph_1, {<}\kappa)$, then $(\dagger)$ and $\mathrm{CH}$ hold in the extension~\cite[Section 1 of Ch. III]{MR1623206}. 

Consider this extension. Note that $\mathrm{CH}$ implies the existence of an $\aleph_2$-Aronszajn tree, and thus $\mathrm{TP}(\aleph_2, \aleph_2)$ fails. Let $\overline{d}$ be a thin $\mathcal{P}_{\aleph_2}\aleph_2$-list with \emph{no} branches. By $(\dagger)$, $\mathrm{Nm}(\aleph_2, F_{\overline{d}})$ must be semiproper, completing the proof.
\end{proof}

We conclude this section with an observation about singular cardinal combinatorics. The \emph{Singular Cardinal Hypothesis (SCH)} asserts that for every strong limit singular cardinal $\mu$, $2^{\mu} = \mu^{+}$. 

\begin{thm}[Sakai~\cite{schwoscale}]\label{thm:s2defi}
 For a strong limit singular cardinal $\mu$ of cofinality $\omega$, if $2^{\mu} > \mu^{+}$, then there exists a stationary subset $S_2\subseteq [\mu^{+}]^{\omega}$ such that for every $x \in \mathcal{P}_{\mu^{+}}\mu^{+}$, $S_2 \cap [x]^{\omega}$ is \emph{not} semistationary. 
\end{thm}

From Theorem~\ref{thm:s2defi}, we can derive the following:

\begin{thm}\label{thm:additional2}
 For a strong limit singular cardinal $\mu$ of cofinality $\omega$, if $2^\mu > \mu^{+}$, then $\mathrm{Nm}(\mu^{+})$ is \emph{not} locally semiproper. 
\end{thm}
\begin{proof}
 If $2^{\mu} > \mu^{+}$, let $S_2$ be a stationary subset of $[\mu]^{\omega}$ from Theorem~\ref{thm:s2defi}. By Theorem~\ref{thm:locsp}, the semistationarity of $S_2$ is not preserved by $\mathrm{Nm}(\mu^{+})$. 
\end{proof}

The reason for introducing this theorem is to demonstrate that many Namba forcings below $\mu$ can still be semiproper, even if $2^{\mu} > \mu^{+}$. 

\begin{prop}\label{prop:schfailmodel}Suppose that $\kappa < \mu$ are strongly compact cardinals. Then there exists a poset that forces:
\begin{enumerate}
 \item $\mu$ is a strong limit singular cardinal of cofinality $\omega$ and $2^{\mu} = \mu^{++}$. In particular, $\mathrm{Nm}(\mu^{+})$ is \emph{not} semiproper.
 \item For every $\mu < \mu$, $\mathrm{Nm}(\aleph_2,\mu)$ is semiproper.
 \item There is a semiproper poset that adds a countable cofinal sequence to $\mu$. 
\end{enumerate}
\end{prop}
\begin{proof}
 By standard arguments, we can obtain a model in which $(\dagger)$ holds and $\mu$ is a measurable cardinal such that $2^{\mu} = \mu^{++}$. Let $P$ be a Prikry forcing over $\mu$. $P$ forces:
\begin{enumerate}
 \item $\mu$ is a strong limit singular cardinal of cofinality $\omega$ and $2^{\mu} = \mu^{++}$. 
 \item ${V}_\mu = V^{V}_{\mu}$. 
\end{enumerate}
For details, we refer to~\cite{Gitik} or~\cite{prikry}. 

Now consider the extension $W$ by $P$. By (2), for every $\mu < \mu$, $\mathrm{Nm}(\aleph_2,\mu)$ is semiproper. 

In particular, $\mathrm{Nm}(\nu)$ is semiproper for all $\nu < \mu$. Fix a sequence $\langle \mu_i \mid i < \omega \rangle$ with $\sup_i \mu_i = \mu$. Let $\mathrm{Nm} = \mathrm{Nm}(\mu_i \mid i < \omega)$ be the set of all trees $p \subseteq \bigcup_{n<\omega}\prod_{i<n}\mu_i$ with the following properties:
\begin{enumerate}
 \item $p$ has a trunk $\mathrm{tr}(p) \in p$. 
 \item For every $t \in p$, $\mathrm{Suc}_{p}(t) = \{\xi < \mu_{|t|} \mid t{^{\frown}}\langle \xi\rangle \in p\}$ is unbounded in $\mu_{|t|}$. 
\end{enumerate}
$\mathrm{Nm}$ is ordered by inclusion. The union $\bigcup \{\mathrm{tr}(p) \mid p \in \dot{G}\}$ is forced to be a new element in $\prod_{i}\mu_i$, where $\dot{G}$ is a $\mathrm{Nm}$-name for the $(W,\mathrm{Nm})$-generic filter. 

For each $t \in p$ with $t \supseteq \mathrm{tr}(p)$, by Lemma~\ref{lem:semiproperchar}, Player II has a winning strategy $\tau_t$ for $\Game_{\mathrm{Gal}}(I_{\mu_{|t|}},\mathrm{Suc}_{p}(t))$. Then, for each countable $M \prec \langle \mathcal{H}_{\theta}, \in, p,\langle\tau_t \mid t \in p\rangle, \langle\mu_i\mid i < \omega\rangle\rangle$, by \cite[Ch.X Theorem 4.12]{MR1623206}, we can find a $q \leq p$ that forces $M \cap \omega_1 = M[\dot{G}] \cap \omega_1$, as required. 
\end{proof}

The forcing $\mathrm{Nm}$, defined above, is an instance of a Namba forcing associated with filter-tagged trees, introduced by Shelah~\cite[Ch.X]{MR1623206}. In the model of Proposition~\ref{prop:schfailmodel}, for every filter-tagged tree $T \subseteq \mathrm{V}_{\mu}$, we can verify that the associated Namba forcing is semiproper. Thus, the semiproperness of many Namba forcings that add a cofinal countable sequence to $\mu$ does not imply $2^{\mu} = \mu^{+}$ for any singular $\mu$ with $\mathrm{cf}(\mu) = \omega$. On the other hand, the following theorem is known.

\begin{thm}[Magidor--Cummings~\cite{MR2811288}]
  For a sequence of regular cardinals $\langle \mu_i \mid i < \omega \rangle$, if $\mathrm{FA}_{\aleph_1}(\mathrm{Nm}(\mu_{i} \mid i <\omega))$ holds, then there is \emph{no} better scale over $\langle \mu_i \mid i < \omega\rangle$. 
\end{thm} 

For the definition of a better scale, we refer to~\cite{MR1838355}. We note that if $\mathrm{SCH}$ fails, then there exists an increasing sequence of regular cardinals $\langle \mu_i \mid i <\omega \rangle$ that carries a better scale~\cite[Ch.II Claim 1.3]{PCF}.

\section{Two-cardinal walks and its applications}\label{sec:twocardinalwalks}

In this section, we introduce two-cardinal walks for Theorem~\ref{maintheorem3}. In this context, a $C$-sequence $\langle C_{x} \mid x \in \mathcal{P}_{\kappa}\lambda\rangle$ is a sequence satisfying the following properties:
\begin{itemize}
    \item $C_{x}$ is a club in $\sup x$ such that $|C_x \cap [\xi_{i},\xi_{i+1})|\leq 1$ if $\max x$ does not exist. Here, $x = \{\xi_{i} \mid i < \mathrm{ot}(x)\}$ is an increasing enumeration.
    \item $C_{x} = \{\max x\}$ if $\sup x \in x$.
\end{itemize}

For $\alpha \in x \in \mathcal{P}_{\kappa}\lambda$, a \emph{walk} from $x$ to $\alpha$ is a sequence $\beta_0,\dots,\beta_{n}$ such that:
\begin{itemize}
    \item $\beta_0 = \sup x$.
    \item $\beta_{i+1}^{p} = \min (C_{x \cap \beta_{i}}\setminus \alpha)$. 
    \item $\beta_{i+1} = \min (x \setminus \beta_{i+1}^{p})$.
    \item $\beta_{n} = \alpha$. 
\end{itemize}
Here, $\sup x = \sup\{\alpha +1 \mid \alpha \in x\}$. We call $\beta_{i+1}^{p}$ a \emph{pre-next step} of $\beta_{i}$. Before proving Theorem~\ref{maintheorem3}, we establish several propositions about these two-cardinal walks.

\begin{prop}
$\beta_{i} = \alpha$ if and only if $\beta_{i}^{p} = \alpha$. 
\end{prop}
\begin{proof}
By definition, we have $\beta_{i} \leq \beta_{i}^{p} \leq \alpha$. Therefore, if $\beta_{i} = \alpha$, then $\beta_{i}^{p} = \alpha$. Conversely, if $\beta_{i}^{p} = \alpha$, then $\beta_{i} = \min x \setminus \alpha = \alpha$, since $\alpha \in x$. 
\end{proof}

A pre-next step determines a true next step in $x$. Conversely, by the definition of the $C$-sequence, we can compute the pre-next step from the true next step. 

\begin{prop}
$\beta_{i+1}^{p} = \max (\beta_{i+1} \cap C_{\beta_{i} \cap x})$. 
\end{prop}
\begin{proof}
The condition $|C_x \cap [\xi_{i},\xi_{i+1})| \leq 1$ ensures this property. 
\end{proof}

\begin{lem}
For all $\alpha \in x \in \mathcal{P}_{\kappa}\lambda$, there exists a walk from $x$ to $\alpha$. 
\end{lem}
\begin{proof}
Since $\alpha \in x$, if $\beta_{0} = \alpha = \max x$, then $\beta_0$ is a walk from $x$ to $\alpha$ with $0$ steps.

By induction, we show that if $\beta_{i} \neq \alpha$, then $\alpha \leq \beta_{i+1} < \beta_{i}$. Suppose that $\beta_i \neq \alpha$, then by the induction hypothesis, $\beta_i > \alpha$. Therefore, $\alpha \in x \cap \beta_{i}$. 

If $x \cap \beta_{i}$ has a maximal element, then $C_{x \cap \beta_{i}} = \{\max x \cap \beta_{i}\}$, so $\alpha \leq \beta_{i+1} = \max x \cap \beta_{i} < \beta_{i}$. 

If $x \cap \beta_{i}$ has no maximal element, then $C_{x \cap \beta_i}$ is a club in $\sup x \cap \beta_{i} \leq \beta_{i}$. Since $\alpha \in x \cap \beta_{i}$ and $x \cap \beta_{i}$ has no maximal element, $\alpha < \sup \beta_{i} \cap x$. So, $\beta_{i+1} = \min (C_{x \cap \beta_{i}}\setminus \alpha) < \sup x \cap \beta_{i} \leq \beta_{i}$, as desired. 

Thus, iteratively taking the next step $\beta_{i + 1}$ generates a descending sequence of ordinals until $\beta_{i + 1} = \alpha$, establishing the existence of a walk from $x$ to $\alpha$. 
\end{proof}

We define $\rho_2$ and $\rho_0$ similarly. $\rho_2(\alpha,x)$ is the number of steps in a walk from $x$ to $\alpha$, and $\rho_0(\alpha,x)$ is $\langle \mathrm{ot}(C_{x \cap \beta_{i}} \cap \alpha) \mid i < \rho_{2}(\alpha,x)\rangle$. Additionally, $\lambda(\alpha,x)$ is defined as $\max\{\sup(C_{\beta_{i}\cap x} \cap \alpha) \mid i < \rho_2(\alpha,x)\}$. The following lemma is analogous to (3) of Lemma~\ref{walk2} and is used in the proof of Theorem~\ref{maintheorem3}. 

\begin{lem}\label{lem:twowalk2}
 For $\lambda(\beta,x) < \alpha < \beta$ and $\alpha,\beta \in x$, $\rho(\alpha,x) = \rho(\beta,x){^{\frown}}\rho(\alpha,x\cap \beta)$. 
\end{lem}
\begin{proof}
A similar proof to (3) of Lemma~\ref{walk2} applies here. 
\end{proof}

For a stationary subset $S \subseteq \lambda$, we say that a $C$-sequence \emph{avoids} $S$ if $S \cap C_{x} = \emptyset$ for all $x \in \mathcal{P}_{\kappa}\lambda$ with $\sup x\not\in x$. 

For $x, y, z \in \mathcal{P}_{\kappa}\lambda$ with $x \cup \{\sup(x)\} \subseteq y$ and $y \cup \{\sup(y)\} \subseteq z$, let $\beta_0> \cdots > \beta_{n}$ and $\beta'_0> \cdots > \beta'_m$ be walks from $z$ to $x$ and $y$, respectively. Define 
\[
[xyz] = \min \{\beta_{0},...,\beta_{n}\} \cap \{\beta_{0}',...,\beta_{m}'\}.
\]
For any other triple $\langle x,y,z\rangle$, define $[xyz] = 0$.

\begin{lem}\label{weaksquarebracket}
 Suppose that $\overline{C}$ avoids a stationary subset $S \subseteq \lambda$. For any $X \in I_{\kappa\lambda}^{+}$, there exists a club $C \subseteq \lambda$ such that 
 \[
 C \cap S \subseteq \{[xyz] \mid \{x,y,z\} \in [X]^{3}_{\subset}\}.
 \]
\end{lem}

\begin{proof}
 Consider the structure $\mathcal{A} = \langle \mathcal{H}_{\theta},\in,\overline{C}, S, X,\kappa,\lambda\rangle$. Let $M \prec \mathcal{A}$ be a countable elementary substructure satisfying:
\begin{itemize}
    \item $\delta = \sup (M \cap \lambda) \in S$.
    \item $M \cap \kappa \in \kappa$ and $|M| < \kappa$.
\end{itemize}
It suffices to find $x, y, z \in X$ such that $[xyz] = \delta$.

Since $\bigcup (\mathcal{P}_{\kappa}\lambda \cap M) = M \cap \lambda$, we can choose $z \in X$ such that $(M \cap \lambda) \cup \{\delta\} \subseteq z$. By assumption, $\delta \not\in C_{x}$ for all $x \in \mathcal{P}_{\kappa}\lambda$ with $\sup x \not\in x$. Note that if $\sup x \in x$, then $C_{x} = \{\max(x)\}$. Thus, $C_x \cap \delta$ is bounded in $\delta$ unless $\sup x = \delta$. This implies $\lambda(\delta,z) < \delta$.

By elementarity of $M$, we can choose $x, y \in X \cap M$ such that:
\begin{enumerate}
    \item $x \cup \{\sup(x)\} \subseteq y$.
    \item $C_{z \cap \delta} \cap [\sup(x),\sup(y)) \neq \emptyset$.
    \item $\lambda(\delta,z) < \sup(x) < \sup(y)$.
\end{enumerate}

From (1) and the choice of $z$, we obtain:
\[
[xyz] = \min \{\beta_i \mid i < \rho_2(\sup x, z)\} \cap \{\beta_j \mid j < \rho_2(\sup y, z)\}.
\]

By (3) and Lemma~\ref{lem:twowalk2}, we have:
\[
\delta \in \{\beta_i \mid i < \rho_2(\sup x, z)\} \cap \{\beta_j \mid j < \rho_2(\sup y, z)\}.
\]

Let $\beta, \gamma$ be the next steps of $\delta$ towards $x$ and $y$, respectively. By (2), we have $\beta \in C_{z \cap \delta} \setminus \sup(x)$ and $\beta \neq \gamma$. Consequently, $[xyz] = \delta$, as desired.
\end{proof}

\begin{proof}[Proof of Theorem~\ref{maintheorem3}]
 Let $S$ be a stationary subset from the assumption. 

 First, we begin by defining a $C$-sequence. By assumption, for each $\alpha \in E_{{<\kappa}}^{\lambda}$, we can choose a club $C_{\alpha}$ such that $S \cap C_\alpha = \emptyset$. For each $x \in \mathcal{P}_{\kappa}\lambda$, define:
 \[
 C_x = \begin{cases}
  C_{\sup x} \cap \mathrm{Lim}(x) & \text{if } \sup x \not\in x, \\
  \{\sup x\} & \text{otherwise}.
 \end{cases}
 \]
 It is easy to see that $\overline{C} = \langle C_x \mid x \in \mathcal{P}_{\kappa}\lambda \rangle$ is a $C$-sequence that avoids $S$.

 Next, we define a coloring on $[\mathcal{P}_{\kappa}\lambda]^{3}_{\subset}$. Let $\langle A_{\xi} \mid \xi < \lambda \rangle$ be a stationary partition of $S$. Define a function $c : [\mathcal{P}_{\kappa}\lambda]^{3} \to \lambda$ witnessing $\mathcal{P}_{\kappa}\lambda \not\to [I_{\kappa\lambda}^{+}]^{3}_{\lambda}$ by setting 
 \[
 c(x,y,z) = \xi \iff [xyz] \in A_\xi.
 \]
 For every $X\in I_{\kappa\lambda}^{+}$, let $C$ be the club given by Lemma~\ref{weaksquarebracket}. For each $\xi$, since $A_\xi \cap C \subseteq \{[xyz] \mid \{x,y,z\} \in [X]^{3}_{\subset}\}$, it follows that $\xi \in c ``[X]^{3}$, as desired. 
\end{proof}

A similar proof establishes the following result:

\begin{thm}\label{thm:twowalks2}
For regular cardinals $\kappa \leq \lambda$, if $\mathrm{Refl}(E_{\omega}^{\lambda},{<}\kappa)$ fails, then $\mathrm{Nm}(\kappa,\lambda)$ is \emph{not} locally semiproper.
\end{thm}

We need the following lemma:

\begin{lem}[Shelah]\label{lem:shelah}
 If $S \subseteq E^{\lambda}_{\omega}$ is a stationary subset and $\langle A_{\xi} \mid \xi < \omega_1 \rangle$ is a stationary partition of $S$, then the set $S_3 = \{ x \in [\lambda]^{\omega} \mid \sup(x) \in A_{x \cap \omega_1} \}$ is stationary.
\end{lem}
\begin{proof}
  See~\cite{MR2387938}. 
\end{proof}

\begin{proof}[Proof of Theorem~\ref{thm:twowalks2}]
Let $S \subseteq E_{\omega}^{\lambda}$ be a stationary subset such that $S \cap \alpha$ is non-stationary for all $\alpha \in E_{<\kappa}^{\lambda}$. By assumption, we have a $C$-sequence $\overline{C} = \langle C_{x} \mid x \in \mathcal{P}_{\kappa}\lambda\rangle$.

Let $\overline{A} = \langle A_{\xi} \mid \xi < \omega_1 \rangle$ be a stationary partition of $S$. Then we have a stationary subset $S_3 \subseteq [\mathcal{H}_{\theta}]^{\omega}$ from Lemma~\ref{lem:shelah}. 

Consider the structure $\mathcal{A} = \langle \mathcal{H}_{\theta},\in,\kappa,\lambda,S,\overline{A},\overline{C}\rangle$. Let $D$ be the set of countable elementary substructures $M\prec \mathcal{A}$. We show that:
\[
\mathrm{Gal}_{\theta}(\mathcal{P}_{\kappa}\lambda,I_{\kappa\lambda},S_3,\mathcal{P}_{\kappa}\lambda) \cap D = S_3^{\uparrow \mathcal{H}_{\theta}} \cap D.
\]

For $\alpha < \lambda$ and $x \in \mathcal{P}_{\kappa}\lambda$ with $\alpha \in x$, let $\beta^{x,\alpha}_0 > \cdots > \beta^{x,\alpha}_{n}$ be a walk from $x$ to $\alpha$. Define a function $f_{\alpha,i}:\mathcal{P}_{\kappa}\lambda \to \omega_1$
by $f_{\alpha,i}(x) = \xi \iff \beta_{i}^{x,\alpha} \in A_{\xi}$
if $\beta_{i}^{x,\alpha}$ exists and belongs to $S$. Otherwise, set $f_{\alpha,i}(x) = 0$. Note that $f_{\alpha,n} \in M$ for all $\alpha \in M \cap \lambda$ and $n < \omega$. Fix $M \in D$ with $M \cap \lambda \in S_3$. 

For any $x \in \mathcal{P}_{\kappa}\lambda$ with $\bigcup (M \cap \mathcal{P}_{\kappa}\lambda) \cup \{\delta\} \subseteq x$, since $\delta \in S$, we have $\lambda(\delta,x) < \delta$. Choose an $\alpha \in (M \cap \lambda) \setminus \lambda(\delta,x)$. Since $\alpha \in x$, we can define a walk from $x$ to $\alpha$. By Lemma~\ref{lem:twowalk2}, there exists some $i$ such that $\beta_{i}^{x,\alpha} = \delta$. Indeed, $i = \rho_2(\delta, x)$. By the definition of $f_{\alpha,i}$, we obtain $f_{\alpha,i}(x) = M \cap \omega_1$. Consequently, 
\[
\mathrm{Sk}(M \cup \{x\}) \cap \omega_1 > M \cap \omega_1,
\]
as desired.
\end{proof}

We conclude this section with the following remark:

\begin{rema}\label{rema:twocardinalwalks}
 Our $C$-sequence is a sequence of club sets; thus, we consider it a ``naive'' $C$-sequence. In \cite[Section 10.4]{MR2355670}, Todor\v{c}evi\'{c} introduced two-cardinal walks to prove Theorem~\ref{todorcevic}. His $C$-sequence over $\mathcal{P}_{\kappa}\lambda$ was defined as a ``copy'' of $C$-sequences below $\kappa$. That is, for every $x \in \mathcal{P}_{\kappa}\lambda$, it was required that 
 \[
 \langle \pi_{x} ``C_{x \cap \alpha} \mid \pi_{x}(\alpha) \in \pi_x ``x \rangle
 \]
 is an ordinal $C$-sequence over $\mathrm{ot}(x)$. Here, $\pi_x$ is the transitive collapse of $x$. 

This definition works well in the case of $\kappa = \nu^{+}$ for some regular $\nu$. However, if $\kappa$ is a limit cardinal, the existence of such a $C$-sequence remains open. Our approach simply uses club sets.

One potential drawback is that our $\rho_0$ function has $\kappa$-many possible values, while Todorčević’s approach restricts it to strictly fewer than $\kappa$. Whether $3$ can be improved to $2$ remains an open question.
\end{rema}

\section{observations and questions}\label{sec:ques}

We conclude this paper with several questions. Questions~\ref{que2} and~\ref{que3} arise since many proofs in this paper required the existence of a thin $\mathcal{P}_{\kappa}\lambda$-list. The connection between tree properties and compactness principles has been extensively studied, yet many open problems remain. 

For instance, to solve Question~\ref{que2}, we aim to extract a suitable combinatorial structure distinct from the tree property (see Proposition~\ref{prop:pfaandstt}) from anti-compactness principles to define a filter. Additionally, it remains an open and widely studied problem whether $\mathrm{ITP}(\aleph_2)$, a strengthening of the tree property $\mathrm{TP}(\aleph_2)$, implies $\mathrm{SCH}$. The stationary subset $S_2$ introduced by Sakai in Lemma~\ref{thm:additional2} was originally formulated to address this question.

\begin{ques}\label{que2}
 Can we remove the assumption that ``$\kappa$ is inaccessible'' from Theorem~\ref{thm:additional1}?
\end{ques}

\begin{ques}\label{que3}
 Does the failure of $\mathrm{SCH}$ at $\mu$ affect the semiproperness of forcings below $\mu$? For example, if $2^{\mu} > \mu^{+}$ for some strong limit singular $\mu$ of cofinality $\omega$, is there a forcing that is $\omega_1$-stationary preserving but \emph{not} semiproper and adds a new countable cofinal subset of $\mu$?
\end{ques}

It appears that extracting combinatorial structures from anti-compactness principles is necessary to resolve the following question. The walks developed in Section~\ref{sec:twocardinalwalks} have not yet been sufficient to answer this problem, as discussed in Remark~\ref{rema:twocardinalwalks}. 

\begin{ques}
 Suppose there exists a stationary subset $S \subseteq E_{<\kappa}^{\lambda}$ such that $S \cap \alpha$ is \emph{non}-stationary for all $\alpha \in E_{<\kappa}^{\lambda}$. Does it follow that $\mathcal{P}_{\kappa}\lambda \not\to [I_{\kappa\lambda}^+]_{\lambda}^{2}$?
\end{ques}

Lastly, we conclude this paper with Figure~\ref{fig:separatedhierarchy}, which illustrates the hierarchy of semiproper Namba forcings. We assume that $\aleph_2 \leq \kappa \leq \lambda = \lambda^{<\kappa}$ and that for all $\mu \in [\aleph_3,\kappa) \cap \mathrm{Reg}$, we have $\mu^{\omega} < \kappa$\footnote{A specific case of interest is $\kappa = \aleph_2$.}. In this hierarchy, $F$ denotes a $\kappa$-complete fine filter. For each implication in the diagram, see~\cite{tsukuura}. 

It is well known that $(\dagger)$-principles have several consequences, one of which is the failure of square sequences. In this paper, we have examined these consequences from the perspective of Namba forcings. By counting how many Namba forcings remain semiproper, we can assess the extent of compactness present in a given model of set theory. Thus, we view the class of Namba forcings as a ``microscope'' for measuring the degree of partial strong compactness of $\kappa$.

The principle $\mathrm{SSR}([\lambda]^{\omega},{<}\kappa)$ asserts that for every semistationary subset $S \subseteq [\lambda]^{\omega}$, there exists an $a \in \mathcal{P}_{\kappa}\lambda$ such that $S \cap [a]^{\omega}$ remains semistationary. Of course, $\mathrm{SSR}$ is $\forall \lambda \geq \aleph_2(\mathrm{SSR}([\lambda]^{\omega},{<}\aleph_2))$.
We observe a significant distinction between $\mathrm{SSR}([2^{\lambda}]^{\omega},{<}\kappa)$ and $\mathrm{SSR}([\lambda]^{\omega},{<}\kappa)$ when analyzed through the lens of Namba forcings, as depicted in Figure~\ref{fig:separatedhierarchy}. Compare it with Figure~\ref{fig:hierarchy}. Note that in Figure~\ref{fig:separatedhierarchy}, every implication denoted by $\to$, except the one between ``$\mathrm{Nm}(\kappa,2^{\lambda})$ is semiproper'' and ``All $\mathrm{Nm}(\kappa,\lambda,F)$ are semiproper,'' is known to be \emph{not} reversible~\cite{tsukuura}. Therefore, we propose the following question:

\begin{ques}
 Is it consistent that $\mathrm{Nm}(\kappa,2^{\lambda})$ is \emph{not} semiproper but there is \emph{no} $\kappa$-complete fine filter $F$ over $\mathcal{P}_{\kappa}\lambda$ such that $\mathrm{Nm}(\kappa,\lambda,F)$ is \emph{non}-semiproper?
\end{ques}

\begin{center}
\begin{figure}[bthp]
\centering
\begin{tikzpicture}
 \node (dagger') at (0,2.5) {$(\dagger)$}; 
 \node (dagger) at (0,2) {$\vdots$}; 
 \node (ssr') at (-4,2.5) {$\mathrm{SSR}$}; 
 \node (ssr) at (-4,2) {$\vdots$}; 

 \draw[<->] (ssr') to (dagger');

\node (nm10) at (0,1) {$\mathrm{Nm}(\kappa,2^{\lambda})$ is sp.};
\node (nm30) at (0,-3) {$\mathrm{Nm}(\kappa,{\lambda})$ is loc. sp.};

\node (nm1) at (0,0) {$\mathrm{Nm}(\kappa, 2^{\lambda})$ is loc. sp.};
\node (nm2) at (0,-1) {All $\mathrm{Nm}(\kappa,\lambda, F)$ are sp.};
\node (nm3) at (0,-2) {$\mathrm{Nm}(\kappa, \lambda)$ is sp.};

\node (nm1') at (3.5,0) {$\mathrm{Nm}(2^{\lambda})$ is loc. sp.};
\node (nm1'0) at (3.5,1) {$\mathrm{Nm}(2^{\lambda})$ is sp.};
\node (nm3') at (3.5,-2) {$\mathrm{Nm}(\lambda)$ is sp.};
\node (nm3'0) at (3.5,-3) {$\mathrm{Nm}(\lambda)$ is loc. sp.};

 \node (nm-) at (3.5,-5) {$\mathrm{Nm}(\kappa)$ is sp.};
 \node (nm'-) at (3.5,-6) {$\mathrm{Nm}(\kappa)$ is loc. sp.};
 
 \node (ssr1) at (-4,0) {$\mathrm{SSR}([2^{\lambda}]^{\omega},{<}\kappa)$};
 \node (ssr2) at (-4,-3) {$\mathrm{SSR}([{\lambda}]^{\omega},{<}\kappa)$};
 \node (ssr-) at (-4,-6) {$\mathrm{SSR}([\kappa]^{\omega},{<}\kappa)$};

 \node (nm0) at (0,-5) {$\mathrm{Nm}(\kappa,\kappa)$ is sp.}; 
 \node (nm01) at (0,-6) {$\mathrm{Nm}(\kappa,\kappa)$ is loc. sp.}; 
 \node (nmdots) at (0,-4) {$\vdots$}; 

 \draw[<->] (nm0) to (nm-); 
 \draw[<->] (nm01) to (nm'-); 
 \draw[->] (nm-) to (nm'-);
 \draw[->] (nm30) to (nm3'0);

\node (sq1) at (6.5,0) {$\lnot \square(2^{\lambda})$};
\node (sq2) at (6.5,-3) {$\lnot \square(\lambda)$};
\node (sq-) at (6.5,-6) {$\lnot \square(\kappa)$};

 \draw[->] (dagger) to (nm10); 
 \draw[->] (nm30) to (nmdots); 
 \draw[->] (nm0) to (nm01); 
 \draw[->] (nmdots) to (nm0); 
 \draw[->] (nm1) to (nm1'); 
 \draw[->] (nm3) to (nm3'); 
 \draw[->] (nm1'0) to (nm1');
 \draw[->] (nm10) to (nm1);

 \draw[->] (nm3) to (nm30);

 \draw[->] (nm1') to node {\small $\times$} (nm3');

 \draw[->] (nm10) to (nm1'0);
 \draw[->] (nm3') to (nm3'0);

\draw[->] (nm1) to (nm2);
\draw[->,transform canvas={xshift=-10pt}] (nm2) to (nm3);

\draw[->] (nm1') to (sq1);
\draw[->] (nm3'0) to (sq2);
\draw[->] (nm'-) to (sq-);

\draw[->,dotted] (nm2) to (sq1);

 \draw[->,dotted,transform canvas={xshift=10pt}] (nm3) to node {\small $\times$} (nm2);

 \draw[<->] (ssr1) to (nm1);
 \draw[<->] (ssr2) to (nm30);
 \draw[<->] (ssr-) to (nm01);

 \node (ssrdots) at (-4,-4.5) {$\vdots$};

 \draw[->] (ssr) to (ssr1); 
 \draw[->] (ssr1) to (ssr2); 
 \draw[->] (ssr2) to (ssrdots); 
 \draw[->] (ssrdots) to (ssr-); 

\end{tikzpicture}
\caption{The dissection of $(\dagger)$-principle focusing on $\kappa$. Dotted arrows were proved in this paper.}
\label{fig:separatedhierarchy}
\end{figure}
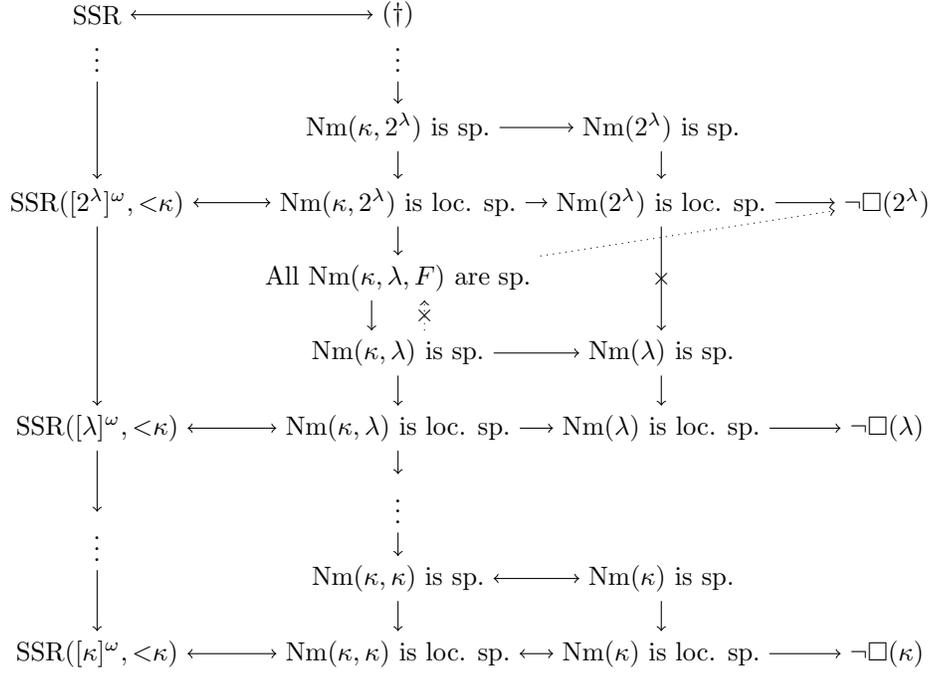
\end{center}

\newpage
   \bibliographystyle{plain}
   \bibliography{ref}

\end{document}